\documentclass[12pt, a4paper]{article}

\usepackage[utf8]{inputenc}
\usepackage[T1]{fontenc}
\usepackage[english]{babel}

\usepackage[margin=3cm]{geometry}

\usepackage[dvipsnames]{xcolor}
\newcommand{\revised}[1]{{#1}}

\usepackage{mathtools, amssymb}
\usepackage{mathrsfs} 

\usepackage{amsthm}
\allowdisplaybreaks
\usepackage{eucal} 
\theoremstyle{plain}
\newtheorem{thm}{Theorem}[subsection]
\newtheorem{lem}[thm]{Lemma}
\newtheorem{prop}[thm]{Proposition}
\newtheorem{cor}[thm]{Corollary}
\newtheorem*{thm*}{Theorem}
\theoremstyle{definition}

\newtheorem{ex}[thm]{Example}
\theoremstyle{remark}
\newtheorem{rmk}[thm]{Remark}

\usepackage{tikz-cd}

\newcommand{\ds}[1]{\mathbf{#1}}

\providecommand{\kat}[1]{\text{\textbf{\textsl{#1}}}}

\newcommand{\Cat}{\kat{Cat}}

\newcommand{\infGrpd}{{\mathcal{S}}}
\newcommand{\infgrpd}{{\mathcal{F}}}
\newcommand{\simplexcat}{\boldsymbol \Delta}

\newcommand{\op}{^{\text{{\rm{op}}}}}

\DeclareMathOperator{\Fun}{Fun}
\DeclareMathOperator{\Map}{Map}
\DeclareMathOperator{\Hom}{Hom}

\renewcommand{\Im}{\operatorname{Im}}

\DeclareMathOperator{\Id}{Id}
\DeclareMathOperator{\id}{id}

\newcommand{\ora}{\overrightarrow}
\newcommand{\xra}{\xrightarrow}
\newcommand{\xla}{\xleftarrow}

\newcommand{\eq}{\simeq}

\newcommand{\ad}{\dashv}

\newcommand{\name}[1]{\ulcorner #1\urcorner}

\newcommand{\Q}{\mathbb{Q}}

\newcommand{\Dec}{\operatorname{Dec}}

\newcommand{\arxiv}[1]{\href{http://arxiv.org/abs/#1}{arxiv:#1}}

\usepackage{hyperref}
\hypersetup{
    pdftitle={Incidence bicomodules, Möbius inversion, and a Rota formula for infinity adjunctions},
    pdfsubject={},
    pdfauthor={Louis Carlier},
    pdfkeywords={},
    colorlinks=true,
    linkcolor=RoyalBlue,
    citecolor=ForestGreen,
    urlcolor=teal,
    breaklinks=true
}

\author{Louis Carlier}
\title{Incidence bicomodules, Möbius inversion, and a Rota formula for infinity adjunctions}
\date{}

\newcommand{\address}{{%
  \bigskip
  \footnotesize
  \textsc{Departament de Matemàtiques \\
  \indent Universitat Autònoma de Barcelona \\
  \indent 08193 Bellaterra (Barcelona), Spain}\par\nopagebreak
  \textit{E-mail address}: 
  \href{mailto:louiscarlier@mat.uab.cat}{\nolinkurl{louiscarlier@mat.uab.cat}}
}}

\begin{document}
\maketitle

\begin{abstract}
    In the same way decomposition spaces, also known as unital 2-Segal spaces, have incidence (co)algebras, and certain relative decomposition spaces have incidence (co)modules, we identify the structures that have incidence bi(co)\-modules:~they are certain augmented double Segal spaces subject to some exactness conditions. We establish a Möbius inversion principle for (co)modules, and a Rota formula for certain more involved structures called Möbius bicomodule configurations.
   The most important instance of the latter notion arises as mapping cylinders of infinity adjunctions, or more generally of adjunctions between Möbius decomposition spaces, in the spirit of Rota's original formula.
\end{abstract}

\setcounter{tocdepth}{1}
\renewcommand{\baselinestretch}{0.5}\normalsize
\tableofcontents
\renewcommand{\baselinestretch}{1.0}\normalsize

\phantomsection
\addcontentsline{toc}{section}{Introduction}
\section*{Introduction}
\label{sec:intro}

The theory of Möbius categories, developed by Leroux~\cite{Leroux76} generalises the theory for locally finite posets~\cite{Rota} and Cartier--Foata finite-decomposition monoids~\cite{CF}, which admit  incidence (co)algebras and a Möbius inversion principle. 
It has recently been generalised to $\infty$-categories and decomposition spaces by Gálvez-Carrillo, Kock, and Tonks~\cite{GKT1, GKT2}.
\revised{(Decomposition spaces are the same thing as the unital $2$-Segal spaces of Dyckerhoff and Kapranov~\cite{DK}.)
}

\revised{
An important tool for computing the Möbius function in the incidence algebra of a locally finite poset is the classical formula of Rota~\cite[Theorem 1]{Rota} which compares the Möbius functions across a Galois connection.
Rota's original work already gave many applications of this formula, and it also features prominently in standard text books such as Aigner~\cite{Aigner} and Stanley~\cite{Stanley1}.
The idea that one should not just work with posets individually, but rather exploit relationships between posets is of course a very modern idea, appealing to any mathematician with a categorical bias.
}

The original goal of the present work, thought to be a routine exercise, was to generalise this formula to $\infty$-adjunctions. It turned out a lot of machinery was required to do this in a satisfactory way, and developing this machinery ended up as a substantial contribution, warranting the present paper: they are general constructions in the theory of decomposition spaces/2-Segal spaces, concerning bicomodules, which are of interest not just in combinatorics but also in representation theory, in connection with Hall algebras.

Before stating and proving this formula for categories, let us recall some definitions.

\subsection*{Incidence coalgebras}

Given a small category $X$, write $X_0$ for its set of objects and $X_1$ for its set of arrows.
Let $\Q_{X_1}$ be the free vector space on $X_1$.
We say a category $X$ is \emph{locally finite} if each morphism $f: x \to z$ in $X$ admits only finitely many two-step factorisations $x \xra{g} y \xra{h} z$.
This condition guarantees that the comultiplication on $\Q_{X_1}$, given by

\begin{align*}
    \Delta:  \Q_{X_1} & \to  \Q_{X_1} \otimes \Q_{X_1} \\
            f & \mapsto \sum_{h g = f} g \otimes h
\end{align*}
is well defined.
The counit $\delta: \Q_{X_1} \to \Q$ is given
by $\delta(\id_x) = 1$, and $\delta(f) = 0$ else. 

The \emph{incidence algebra} $\mathcal{I}_X$ is the linear dual, ($\textrm{Lin}(\Q_{X_1},\Q),\ast,\delta)$ with the convolution product:
\[
(\alpha \ast \beta)(f)
=
\sum_{h g = f}
\alpha(g) \, \beta(h),
\]
where $\alpha, \beta \in \mathcal{I}_X$ and $f \in \Q_{X_1}$.

The \emph{zeta function} $\zeta_X: \Q_{X_1} \to \Q$ is defined by
$\zeta_X(f) = 1$ for all $f \in \Q_{X_1}$. 

Define $\Phi_{\text{even}}: \Q_{X_1} \to \Q$ to be the number of even-length factorisations of a morphism, without identities, and $\Phi_{\text{odd}}: \Q_{X_1} \to \Q$ to be the number of odd-length factorisations, without identities.
A category is \emph{Möbius}~\cite{Leroux76} if it is locally finite and $\Phi_{\text{even}}$ and $\Phi_{\text{odd}}$ are finite.

\begin{thm*}[Content, Lemay, Leroux \cite{CLL}]
    If $X$ is a Möbius category then the zeta function is invertible, and the inverse, called the Möbius function, is given by $\mu = \Phi_{\operatorname{even}} - \Phi_{\operatorname{odd}}$.
\end{thm*}

\revised{
Examples of Möbius categories are locally finite posets and monoids with the finite-decompositions property, and this theorem generalises similar theorems for these more specialised settings.
}


\subsection*{Rota formula for categories}
\label{sub:rotacat}

A classical formula due to Rota~\cite{Rota} compares the Möbius functions of two posets related by a \emph{Galois connection}.
The following generalisation of Rota's formula to Möbius categories is both natural and straightforward (but seems not to have been made before). 

\begin{thm*}[Rota formula for Möbius categories]\label{thm:Rotacat}
Let $X$ and $Y$ be Möbius categories, and let $F:X \rightleftarrows Y:G$ be an adjunction, $F \ad G$. 
Then for all $x \in X$, $y \in Y$,
\[
\sum_{\mathclap{\substack{x' \in X \\ f: x \to x' \\ F x' = y}}} \mu_X(f)
=
\sum_{\mathclap{\substack{y' \in Y\\ g:y' \to y \\ G y' = x}}} \mu_Y(g).
\]
\end{thm*}

The reader is not expected to read the following elementary proof, but only notice that it looks like an associativity formula for a convolution product, except that the arrows live in different categories. 

\begin{proof}[Proof of the Rota formula]
\begin{align*}
\smashoperator{\sum_{\substack{x' \in X \\ f: x \to x' \\ F x' = y}}} \mu_X(x \xra{f} x')
    &\stackrel{(1)}{=} \sum_{\substack{x' \in X \\ f: x \to x' \\ h: F x' \to y}} \mu_X(x \xra{f} x') \delta_Y(F x' \xra{h} y) \\
    &\stackrel{(2)}{=} \sum_{\substack{x' \in X \\ f: x \to x' \\ h: F x' \to y}} \mu_X(x \xra{f} x') \left( \sum_{\substack{y' \in Y \\ g:y' \to y \\ h': F x' \to y' \\\textrm{s.t. } h = gh'}} \zeta_Y(F x' \xra{h'} y') \mu_Y(y' \xra{g} y) \right) \\
    &\stackrel{(3)}{=} \sum_{\substack{x' \in X, y' \in Y \\ f: x \to x' \\ h: F x' \to y\\ g:y' \to y \\ h': F x' \to y' \\\textrm{s.t. } h = gh'}} \mu_X(x \xra{f} x') \zeta_Y(F x' \xra{h'} y') \mu_Y(y' \xra{g} y) \\
    &\stackrel{(4)}{=} \sum_{\substack{x' \in X, y' \in Y \\ g: y' \to y \\ k: x \to G y'\\ f:x \to x' \\ k': x' \to G y' \\\textrm{s.t. } k = k'f}} \mu_X(x \xra{f} x') \zeta_X(x' \xra{k'} G y') \mu_Y(y' \xra{g} y) \qquad \text{ by adjunction} \\
    &\stackrel{(5)}{=} \sum_{\substack{y' \in Y \\ g: y' \to y \\ k: x \to G y'}} \left( \sum_{\substack{x' \in X \\f:x \to x' \\ k': x' \to G y' \\\textrm{s.t. } k = k'f}} \mu_X(x \xra{f} x') \zeta_X(x' \xra{k'} G y') \right) \mu_Y(y' \xra{g} y) \\
    &\stackrel{(6)}{=} \sum_{\substack{y' \in Y \\ g: y' \to y \\ k: x \to G y'}} \delta_X(x \xra{k} G y') \mu_Y(y' \xra{g} y) 
    \stackrel{(7)}{=} \sum_{\substack{y' \in Y\\ g:y' \to y \\ G y' = x}} \mu_Y(y' \xra{g} y).
\end{align*}
\end{proof}
In the main result of the present paper, Theorem~\ref{rotaformulabicomod}, we write this formula as
\[
    \mu_{X} \star_l \delta_Y = \delta_{X} \star_r \mu_Y,
\]
with the following more conceptual proof
\[
      \mu_X \star_l \delta_Y 
        \stackrel{(2)}{=}  \mu_X \star_l (\zeta \star_r \mu_Y) 
        \stackrel{(3-4-5)}{=} (\mu_X \star_l \zeta) \star_r \mu_Y 
        \stackrel{(6)}{=} \delta_{X} \star_r \mu_Y,
\]
referring to certain left and right convolution actions.
\revised{
The important insight here, which is due to Aguiar and Ferrer~\cite{AF}, is that the `mixed arrows' which appear in the middle factors (those of the form $Fx\to y$, which in the crucial step of the proof are reinterpreted as $x\to Gy$ by adjunction) belong to a \emph{bimodule}: they are acted upon from the left by
arrows in the category $X$ and from the right by arrows in $Y$ (see Example~\ref{exFG} for an explicit description). The long complicated sums in the proof are thus condensed into convolution actions from the left and right, denoted $\star_l$ and $\star_r$. Aguiar and Ferrer~\cite{AF} established a bimodule proof of the Rota formula in the setting of posets.
}

Two ingredients are necessary to make sense of the pleasing convolution proof above: one is to exhibit the data necessary to induce bicomodules and establish that adjunctions constitute an example. \revised{This already accounts for equalities (3-4-5) in the proof.}
The other is to establish a Möbius inversion principle for (co)modules (a notion which has not previously been considered in the literature, to the knowledge of the author), \revised{to account for the equalities (2) and (6).}

\revised{
In fact, as the notions and arguments become increasingly abstract and conceptual, it is natural to ask for further generalisation. In this work we take three considerable abstraction steps (beyond passing from posets to categories, which is already a fruitful step). First, we pass from categories and adjunctions to $\infty$-categories and $\infty$-adjunctions. Any $\infty$-adjunction defines a bicomodule in our sense. This step in itself is not so easy to justify from the viewpoint of combinatorics, but the homotopy content inherent in $\infty$-categories is important since already classical combinatorial structures have symmetries, and these can be handled more conveniently with groupoids than with sets, as advocated by Baez--Dolan~\cite{BD}, Gálvez-Kock-Tonks~\cite{GKT:combinatorics} and others. (This aspect will not be of importance in the present contribution, though.) Second, we pass from $\infty$-categories to decomposition spaces (also called $2$-Segal spaces) and introduce a notion of adjunction for them. This step has an important     combinatorial motivation, because many combinatorial coalgebras admit a natural realisation as incidence coalgebras of decomposition spaces which are not posets or categories.
An important example is the coalgebra of all finite posets, which will serve as a running example in this paper. The final abstraction step consists in noticing that the abstract Rota formula works equally well for certain bicomodules which do {\em not} come from adjunctions or $\infty$-adjunctions. In fact the running example chosen to illustrate the theory is of this type: it is a certain bicomodule interpolating between the decomposition space of finite sets and the decomposition space of finite posets. The outcome is the formula $\mu(P) = (-1)^n$ for the Möbius function of a poset with $n$ elements. This formula is well known (see for example \cite{ABS}) but its derivation via a Rota formula is new and interesting, since the coalgebra of finite posets is not the incidence coalgebra of a locally finite poset or Möbius category.

Finally, a word should be said about the objective approach, an important aspect of the decomposition-space viewpoint on incidence algebras and Möbius inversion. The point here is to lift combinatorial identities to bijections of sets, and more generally equivalences of $\infty$-groupoids. Specifically, the Möbius-inversion formula, which classically is an equation in vector spaces, is lifted to an equivalence between $\infty$-groupoids defining spans, whose homotopy cardinality are the linear maps in question (the introductions of \cite{GKT1} and \cite{GKT:combinatorics} contain further motivation). The present work inscribes itself in this tradition, seeking to find objective structures for bicomodules and the Rota formula. However, it must be admitted that the objective level is not fully achieved in this paper. On one hand, the Möbius formula obtained for comodules is not directly realised as the homotopy cardinality of an equivalence: it has been found necessary here to take homotopy cardinality a little bit earlier in the constructions, so that the final arguments take place at the vector space level. This is due to the increased complexity compared to the plain Möbius-inversion formula of \cite{GKT2}, where an even-odd splitting could be found for the single decomposition space involved.
In the present situation, \emph{two} decomposition spaces are involved, and the even-odd splitting at the objective level could not be found. Furthermore, the objective analogue of bicomodules given here is not fully satisfactory from the homotopy viewpoint. While it is shown to induce bicomodules up to homotopy, the \emph{coherence} of this up-to-homotopy structure is not established in this work, and would seem to require considerable further efforts, in the line of coherence proofs given by \cite{GKT1} and \cite{Penney}. Further discussion is included in the main text. The justification for not establishing coherence in the present contribution is that it is not necessary for the sake of taking cardinality, as required anyway in the final constructions for the Rota formula established. 
}

Although the motivation and the statement of the theorem belongs to combinatorics, the setting for this work and the tools employed are from simplicial homotopy theory, in the style of \cite{GKT1, GKT2}, working with $\infty$-groupoids, homotopy pullbacks, mapping spaces, fibrations and fibre sequences. One technical novelty compared to \cite{GKT1} and \cite{GKT2} is that we exploit general simplicial maps between decomposition spaces, not just culf ones, and introduce the notion of adjunction between decomposition spaces. Another is that the notion of mapping cylinder is exploited systematically: on one hand locally to model the shapes needed to index the various configurations, and on the other hand globally, as infinity mapping cylinders.

\subsection*{Outline of the paper}
\label{sub:outline}

We begin in Section~\ref{sec:preliminaries} with a brief review of needed notions from the theory of $\infty$-categories, with an emphasis on decomposition spaces.

In Section~\ref{sec:bicomodules}, following Walde~\cite{Walde} and Young~\cite{Young}, we first explain how to obtain a comodule in the context of decomposition spaces.

\medskip

\noindent
{\bf Proposition~\ref{comodule}.} {\em
If $f: C \to X$ is a culf map between two simplicial $\infty$-groupoids such that $C$ is Segal and $X$ is a decomposition space, then the span
\[
    C_{0}  \xleftarrow{d_1} C_{1} 
    \xrightarrow{(f_{1},d_0)} X_{1} \times C_{0}
\]
induces on the slice $\infty$-category $\infGrpd_{/C_0}$ the structure of a left $\infGrpd_{/X_1}$-comodule (at the $\pi_0$ level), and the span 
\[
    C_{0}  \xleftarrow{d_0} C_{1} 
    \xrightarrow{(d_1,f_{1})} C_{0} \times X_{1}.
\]
induces on $\infGrpd_{/C_0}$ the structure of a right $\infGrpd_{/X_1}$-comodule (at the $\pi_0$ level).
}
\medskip

The data needed to obtain a comodule is called a \emph{comodule configuration}.
\revised{In order to obtain a bicomodule structure, we first need an augmented bisimplicial $\infty$-groupoid Segal in each direction. We furthermore require this bisimplicial $\infty$-groupoid to be stable}, see Section~\ref{sub:stability}.
This stability condition is a pullback condition on certain squares, and is a $\infty$-categorical reformulation of the notion of Bergner--Osorno--Ozornova--Rovelli--Scheimbauer~\cite{BOORS}, suitable for $\infty$-groupoids.

\medskip
\noindent
{\bf Theorem~\ref{thm:bicomodule}.} {\em
Let $B$ be an augmented stable double Segal space, and such that the augmentation maps are culf.
Suppose moreover $X:=B_{\bullet,-1}$ and $Y:=B_{-1,\bullet}$ are decomposition spaces.
Then the spans
\[B_{0,0}  \xleftarrow{e_1} B_{1,0} 
    \xrightarrow{(u,e_0)} X_{1} \times B_{0,0}
\]
and
\[B_{0,0}  \xleftarrow{d_0} B_{0,1} 
    \xrightarrow{(d_1,v)} B_{0,0} \times Y_{1}
\]
induce on $\infGrpd_{/B_{0,0}}$ the structure of a bicomodule over $\infGrpd_{/X_1}$ and $\infGrpd_{/Y_1}$ (at the $\pi_0$ level).

}
\medskip

An augmented bisimplicial $\infty$-groupoid satisfying the conditions of the theorem is called a \emph{bicomodule configuration}.
\revised{
We describe as an example the bicomodule configuration of layered sets and posets, treated in more details in \cite{Ca:mdrs}.
}

In Section~\ref{sec:adj}, \revised{we introduce the notion of correspondence of decomposition spaces: it is a decomposition space $\mathcal{M}$ with a map $\mathcal{M} \to \Delta^1$.}
We show that any correspondence of decomposition spaces gives rise to a bicomodule configuration.
We then introduce the notion of \emph{cartesian} and \emph{cocartesian fibration} of decomposition spaces, adapting a homotopy-invariant definition for $\infty$-categories which can be found in \cite{AF}.
\revised{They give rise to left and right pointed comodule configurations.}
We define an \emph{adjunction} between decomposition spaces $X$ and $Y$ to be a simplicial map between decomposition spaces $p: \mathcal{M} \to \Delta^1$ which is both a cartesian and a cocartesian fibration, equipped with equivalences $X \eq \mathcal{M}_{\{0\}}$ and $Y \eq \mathcal{M}_{\{1\}}$.
\revised{Adjunctions give rise to bicomodule configurations with two pointings.}

In Section~\ref{sec:mobcomod}, we define left and right convolution actions $\star_l$ and $\star_r$ dual to the comodule structures.
The following is a consequence of Theorem~\ref{thm:bicomodule}.

\medskip
\noindent
{\bf Corollary~\ref{associativityconvol}.} {\em
Given a bicomodule configuration, the left and right convolutions satisfy the associative law
\[
    \alpha \star_l (\theta \star_r \beta)
     \eq (\alpha \star_l \theta) \star_r \beta.
\]

}

We then establish in Section~\ref{sub:phifunctors} a Möbius inversion principle for complete comodules. 
Let $C \to Y$ be a right comodule configuration such that the simplicial $\infty$-groupoid $C$ is augmented and with 
new bottom degeneracies $s_{-1}: C_{n-1} \to C_{n}$ which are sections to $d_0$. 
We say it is \emph{complete} (Section~\ref{sub:complete}) if the sections $s_{-1}$ are monomorphisms.

For a complete decomposition space $Y$, let $\ora{Y_n}$ denote the full subgroupoid of simplices with all principal edges nondegenerate. The spans
$Y_1 \xla{d_1^{n-1}} \ora{Y_n} \rightarrow 1$
define linear functors, the \emph{Phi functors}
$
    \Phi_n: \infGrpd_{/Y_1} \to \infGrpd
$.
We also put $\displaystyle \Phi_{\text{even}} := \sum_{n \text{ even}} \Phi_{n}$, and $\displaystyle \Phi_{\text{odd}} := \sum_{n \text{ odd}} \Phi_{n}$.

The zeta functor 
$\zeta^{C}: \infGrpd_{/C_{0}} \to \infGrpd$
is the linear functor defined by the span $C_{0} \xleftarrow{=} C_{0} \xrightarrow{} 1$,
and
$\delta^{R}: \infGrpd_{/C_{0}} \to \infGrpd$
is the linear functor given by the span
$C_{0} \xleftarrow{s_{-1}} C_{-1} \xrightarrow{} 1$.
We define $\delta^{L}$ similarly for left comodule configurations.

\medskip
\noindent
{\bf Theorem~\ref{thm:Mobinversion} and \ref{thm:Mobinversionleft}.} {\em
    Given $C \to Y$ a complete right comodule configuration and $D \to X$ a complete left comodule configuration, then
    \[\zeta^{C} \star_r \Phi_{\text{even}}^Y \eq \delta^{R} + \zeta^{C} \star_r \Phi_{\text{odd}}^Y,\]
    \[\Phi_{\text{even}}^X \star_l \zeta^{D}  \eq \delta^{L} +  \Phi_{\text{odd}}^X \star_l \zeta^{D}.\]

}

In Section~\ref{sub:Mobbicomod}, we establish a Möbius inversion principle at the algebraic level. To this end, we need to impose some finiteness conditions in order to take homotopy cardinality.
Define the \emph{Möbius functions} as the homotopy cardinalities $|\mu^Y| := |\Phi^Y_{\text{even}}| - |\Phi^Y_{\text{odd}}|$
and $|\mu^X| := |\Phi^X_{\text{even}}| - |\Phi^X_{\text{odd}}|$.

\medskip
\noindent
{\bf Theorem~\ref{cardmobform}.} {\em
Given $C \to Y$ a  right Möbius comodule configuration and $D \to X$ a left Möbius comodule configuration,
    \[
    |\zeta^{C}| \star_r |\mu^Y| = |\delta^{R}|, \qquad  |\mu^X| \star_l |\zeta^{D}|  = |\delta^{L}|.\]

}

Finally we can extend the Rota formula to bicomodules with Möbius inversion for both comodules, called \emph{Möbius bicomodule configurations}. Combining Proposition~\ref{associativityconvol} and Theorem~\ref{cardmobform}, we obtain the main theorem of the present paper:

\medskip
\noindent
{\bf Theorem~\ref{rotaformulabicomod}.} {\em
Given a Möbius bicomodule configuration $B$ with $X := B_{\bullet, -1}$ and $Y:=B_{-1, \bullet}$, we have
\[
    |\mu^{X}| \star_l |\delta^R| = |\delta^{L}| \star_r |\mu^Y|,
\]
where $\delta^R$ is the linear functor given by the span
\[
    B_{0,0} \xleftarrow{} X_{0} \xrightarrow{} 1,
\]
and $\delta^{L}$ is the linear functor given by the span
\[
    B_{0,0} \xleftarrow{} Y_{0} \xrightarrow{} 1.
\]

}

The motivating example, treated in Section~\ref{sub:runexMob}, shows that any (co)cartesian fibration $p: \mathcal{M} \to \Delta^1$ such that $\mathcal{M}$ is a complete decomposition space gives rise to a complete left (or right) comodule configuration:

\medskip
\noindent
{\bf Theorem~\ref{runexRota}.} {\em
     Given an adjunction of decomposition spaces in the form of a bicartesian  fibration $p: \mathcal{M} \to \Delta^1$, suppose moreover that $\mathcal{M}$ is a Möbius decomposition space. Then the bicomodule configuration extracted from this data is Möbius. In particular, we have the Rota formula for the adjunction $p$:
\[
    |\mu^{X}| \star_l |\delta^R| = |\delta^{L}| \star_r |\mu^Y|.
\]

}

When specialised to the case of a classical adjunction between 1-categories, this is the classical Rota formula from page~\pageref{thm:Rotacat}.

\revised{
Finally, the bicomodule configuration of layered sets and posets defined in Section~\ref{sec:bicomodules} is Möbius, and we apply the generalised Rota formula to compute the Möbius function of the incidence algebra of the decomposition space of finite posets.

\medskip
\noindent
{\bf Theorem~\ref{formula}} (\cite[Theorem~3.4]{Ca:mdrs}){\bf .} {\em
    The Möbius function of the incidence algebra of the decomposition space $X$ of finite posets is 
    \[
    \mu(P) = 
    \begin{cases}
       (-1)^n &\text{ if $P \in X_1$ is a discrete poset with $n$ elements,}\\
         0    &\text{ else.}
    \end{cases}
    \]
}

It is shown in \cite{Ca:mdrs} how this result implies similar results for any directed restriction species or free operad.
}

\subsection*{Acknowledgements}
 The author would like to thank Joachim Kock not only for suggesting to investigate this Rota formula but also for help and support all along the project, and also Christian Sattler for useful remarks.
The author is grateful to Kurusch Ebrahimi-Fard and Yannic Vargas, who independently pointed out the work of Aguiar and Ferrer.
Finally the author wish to thank the anonymous referee for a thorough and constructive report that led to many expository improvements.
The author was supported by PhD grant attached to MTM2013-42293-P, and grant number MTM2016-80439-P of Spain.

\section{Preliminaries}
\label{sec:preliminaries}

We work in the $\infty$-category of $\infty$-groupoids, denoted $\infGrpd$, following the notation of \cite{GKT1}. Our $\infty$-categories are quasi-categories; 
the theory of quasi-categories has been substantially developed by Joyal \cite{Joyal02, Joyal08} and Lurie \cite{Lurie}.
An $\infty$-groupoid is an $\infty$-category in which all morphisms are invertible. They are precisely Kan complexes: simplicial sets in which every horn admits a filler (and not only the inner ones).

\subsection{Pullbacks and fibres, slices and linear functors}
\label{sub:pbk}
The main tool used throughout this paper are pullbacks.
We use the following standard lemma many times.

\begin{lem}[{\cite[Lemma 4.4.2.1]{Lurie}}]\label{prismlemma}
  Given a prism diagram of $\infty$-groupoids
  \begin{center}
    \begin{tikzcd}
        X   \arrow[r, ""] \arrow[d, ""]
          & X' \arrow[r, ""]\arrow[d, ""] \arrow[dr, phantom, "\lrcorner", very near start]
          & X''\arrow[d, ""]\\
        Y \arrow[r, ""'] & Y' \arrow[r, ""] & Y''
    \end{tikzcd}
  \end{center}
    in which the right-hand square is a pullback. 
    Then the outer rectangle is a pullback if and only if the left-hand square is.
\end{lem}


Given a map of $\infty$-groupoids $p:X \to S$, an object $s \in S$, the \emph{fibre} $X_s$ of $p$ over $s$ is the pullback
\begin{center}
    \begin{tikzcd}
        X_s \arrow[r, ""] 
           \arrow[d, ""'] 
           \arrow[dr, phantom, "\lrcorner", very near start]
             & X \arrow[d, "p"]\\
        1 \arrow[r, "\name{s}"'] & S.
    \end{tikzcd}
\end{center}

A map of $\infty$-groupoids is a \emph{monomorphism} when its fibres are $(-1)$-groupoids, that is, are either empty or contractible.
If $f: X \to Y$ is a monomorphism, then there is a complement $Z := Y \backslash X$ such that $Y \eq X + Z$; a monomorphism is essentially an equivalence from X onto some connected components of Y.

Recall that the objects of the slice $\infty$-category $\infGrpd_{/I}$ are maps of $\infty$-groupoids with codomain $I$.
Pullback along a morphism $f: J \to I$, defines an functor $f^*: \infGrpd_{/I} \to \infGrpd_{/J}$. This functor is right adjoint to the functor $f_!: \infGrpd_{/J} \to \infGrpd_{/I}$ given by post-composing with $f$.
A span
$I \xleftarrow{p} M \xrightarrow{q} J$
induces a functor between the slices by pullback and postcomposition
\[\infGrpd_{/I} \xra{p^*} \infGrpd_{/M} \xra{q_!} \infGrpd_{/J}.\]
A functor is \emph{linear} if it is homotopy equivalent to a functor induced by a span.
The following Beck-Chevalley rule holds for $\infty$-groupoids:
for any pullback square
    \begin{center}
        \begin{tikzcd}
            J \arrow[r, "f"] 
               \arrow[d, "p"'] 
               \arrow[dr, phantom, "\lrcorner", very near start]
               & I \arrow[d, "q"]\\
            V \arrow[r, "g"'] & U,
        \end{tikzcd}
    \end{center}
the functors
$
    p_!f^*, g^*q_! : \infGrpd_{/I} \to \infGrpd_{/V}
$
are naturally homotopy equivalent (see \cite{GHK} for the technical details regarding coherence of these equivalences).
By the Beck-Chevalley rule, the composition of two linear functors is linear.
For an extended treatment of linear functors and homotopy linear algebra, we refer to \cite{GKT:HLA}.


\subsection{Segal spaces and decomposition spaces}

We consider the functor $\infty$-category
\[
    \Fun{(\simplexcat\op, \infGrpd)}
\]
whose objects are \revised{\emph{simplicial $\infty$-groupoids}, that is} functors from the $\infty$-category $\simplexcat\op$ to the $\infty$-category $\infGrpd$.



A simplicial $\infty$-groupoid $X$ is called a \emph{Segal space} if the following squares are pullbacks, for all $n > 0$:
\begin{center}
    \begin{tikzcd}
        X_{n+1}  \arrow[dr, phantom, "\lrcorner", very near start] \arrow[r, "d_0"] 
           \arrow[d, "d_{n+1}"'] & X_{n}  \arrow[d, "d_n"]\\
         X_{n} \ \arrow[r, "d_0"'] & X_{n-1}.
    \end{tikzcd}
\end{center}

The simplex category $\simplexcat$ has an active-inert factorisation system.
A morphism $[m] \to [n]$ is \emph{active} (also called \emph{generic}) if it preserves endpoints: $g(0) = 0$, $g(m) = n$. A morphism is \emph{inert} (also called \emph{free}) if it is distance preserving: $f(i+1) = f(i) + 1$, for $0 \le i \le m - 1$.
The active maps are generated by the codegeneracy maps and the inner coface maps, and the inert maps are generated by the outer coface maps $d^\bot := d^0$ and $d^\top := d^n$.

    A \emph{decomposition space} $X: \simplexcat\op \to \infGrpd$ is a simplicial $\infty$-groupoid such that the image of any pushout diagram in $\simplexcat$ of an active map $g$ along an inert map $f$ is a pullback of $\infty$-groupoids.
It is enough to check that the following squares are pullbacks, where $0 \le k \le n$:
  \begin{center}
    \begin{tikzcd}
        X_{n+1}   \arrow[r, "s_{k+1}"] 
                  \arrow[d, "d_\bot"']
                  \arrow[dr, phantom, "\lrcorner", very near start]
          & X_{n+2} \arrow[d, "d_\bot"] \\
        X_n \arrow[r, "s_k"'] & X_{n+1},
    \end{tikzcd} \!\!\!
    \begin{tikzcd}
       X_{n+2} \arrow[d, "d_\bot"] 
          & X_{n+3} \arrow[l, "d_{k+2}"']
                    \arrow[d, "d_\bot"]
                    \arrow[dl, phantom, "\llcorner", very near start] \\
         X_{n+1}  & X_{n+2} \arrow[l, "d_{k+1}"],
    \end{tikzcd} \!\!\!
    \begin{tikzcd}
        X_{n+1}   \arrow[r, "s_k"] 
                  \arrow[d, "d_\top"']
                  \arrow[dr, phantom, "\lrcorner", very near start]
          & X_{n+2} \arrow[d, "d_\top"]  \\
        X_n \arrow[r, "s_k"'] & X_{n+1},
    \end{tikzcd} \!\!\!
    \begin{tikzcd}
        X_{n+2} \arrow[d, "d_\top"] 
          & X_{n+3} \arrow[l, "d_{k+1}"']
                    \arrow[d, "d_\top"]
                    \arrow[dl, phantom, "\llcorner", very near start] \\
         X_{n+1}  & X_{n+2} \arrow[l, "d_{k+1}"].
    \end{tikzcd}
  \end{center}

The notion of decomposition space was introduced by Gálvez-Carrillo, Kock, and Tonks \cite{GKT1}, and independently by Dyckerhoff and Kapranov \cite{DK} under the name unital $2$-Segal space. It can be seen as an abstraction of posets. 
\revised{The equivalence of the two notions follows from the pullback formulation of 2-Segal spaces given in Proposition 2.3.2 of \cite{DK}.}
It is precisely the condition required to obtain a counital coassociative comultiplication on $\infGrpd_{/X_1}$, see also \cite{Penney} for the exact role played by the decomposition space condition.
Since the motivation in the present paper comes from combinatorics, we follow the terminology of \cite{GKT1}; for a survey motivated by combinatorics, see \cite{GKT:combinatorics}.

\begin{prop}[{\cite[Proposition 2.3.3]{DK}}, {\cite[Proposition 3.5]{GKT1}}]
    Every Segal space is a decomposition space.
\end{prop}

There are plenty of examples of decomposition spaces which
are not Segal, e.g. Schmitt's Hopf algebra of graphs, which is a running example in \cite{GKT:combinatorics}.


\subsection{Incidence coalgebras and culf functors}
\label{sub:incidencecoalg}

For any decomposition space $X$, we get an incidence coalgebra \cite{DK}, \cite{GKT1}.
The span 
$
X_1 \xleftarrow{d_1} X_2 \xrightarrow{(d_2,d_0)} X_1 \times X_1
$
defines a linear functor, the \emph{comultiplication}:
\begin{align*}
    \Delta:  \infGrpd_{/X_1} & \to  \infGrpd_{/X_1 \times X_1} \\
            (T \xra{t} X_1) & \mapsto (d_2,d_0)_{!} \circ d_1^{*}(t).
\end{align*}
The span 
$X_1 \xleftarrow{s_0} X_0 \xrightarrow{z} 1$
defines a linear functor, the \emph{counit}:
\begin{align*}
    \delta:  \infGrpd_{/X_1} & \to  \infGrpd \\
            (T \xra{t} X_1) & \mapsto z_{!} \circ s_0^{*}(t).
\end{align*}
The up-to-coherent-homotopy coassociativity follows from the decomposition space axioms, see \cite[\S 5 and 7]{GKT1} or \cite[\S 4.3]{Penney} for a proof.
We obtain a coalgebra $(\infGrpd_{/X_1}, \Delta, \delta)$ called the \emph{incidence coalgebra}.

The $\infty$-category $\infGrpd_{/I}$ plays the role of the vector space with basis $I$. The presheaf category $\infGrpd^I$ can be considered the linear dual of the slice category $\infGrpd_{/I}$ (see \cite{GKT:HLA} for the precise statements and proofs).
\revised{A span $I \xla{} M \xra{} J$ defines both a linear functor 
$\infGrpd_{/I} \xra{} \infGrpd_{/J}$
and the dual linear functor $\infGrpd^J \to \infGrpd^I$.
}
If X is a decomposition space, the coalgebra structure on $\infGrpd_{/X_1}$ therefore induces an algebra structure on $\infGrpd^{X_1}$. In details,
the \emph{convolution product} of two linear functors
$
    F,G: \infGrpd_{/X_1} \to \infGrpd,
$
given by the spans 
 $X_1 \xla{} M \xra{} 1$ and $X_1 \xla{} N \xra{} 1$, 
 is the composite of their tensor product $F \otimes G$ with the comultiplication:
\[F \ast G: \infGrpd_{/X_1} \xra{\Delta} \infGrpd_{/X_1} \otimes \infGrpd_{/X_{1}} \xra{F \otimes G} \infGrpd \otimes \infGrpd \eq \infGrpd,\]
where the tensor product $F \otimes G$ is given by the span 
$X_1 \times X_{1} \xla{} M \times N \xra{} 1$.
The neutral element for convolution is 
\[
\delta: \infGrpd_{/X_1} \to \infGrpd
\]
defined by the span 
$X_1 \xla{s_0} X_0 \xra{} 1.$

A map $f: X \to Y$ of simplicial spaces is \emph{cartesian} on an arrow $[n] \to [k]$ in $\simplexcat$ if the naturality square for $F$ with respect to this arrow is a pullback.
It is called a \emph{right fibration} if it is cartesian on $d_\bot$ and on all active maps, and is called a \emph{left fibration} if it is cartesian on $d_\top$ and on all active maps.

A simplicial map $f: X \to Y$ is \emph{conservative} if it is cartesian with respect to codegeneracy maps
\begin{center}
    \begin{tikzcd}
        X_n  \arrow[dr, phantom, "\lrcorner", very near start] \arrow[r, "s_i"] 
           \arrow[d, "f_n"'] & X_{n+1}  \arrow[d, "f_{n+1}"]\\
         Y_{n} \ \arrow[r, "s_i"'] & Y_{n+1} 
    \end{tikzcd},
    $\qquad 0 \le i \le n$.
\end{center}
It is \emph{ulf} (unique lifting of factorisations) if it is cartesian with respect to inner coface maps

\begin{center}
    \begin{tikzcd}
        X_{n+1}  \arrow[dr, phantom, "\lrcorner", very near start] 
           \arrow[d, "f_{n+1}"'] & X_{n+2} \arrow[l, "d_{i+1}"'] \arrow[d, "f_{n+2}"]\\
         Y_{n+1}  & Y_{n+2} \arrow[l, "d_{i+1}"']
    \end{tikzcd},
    $\qquad 0 \le i \le n$.
\end{center}
We write \emph{culf} for conservative and ulf, that is cartesian on all active maps.
The culf functors induce coalgebra homomorphisms between the incidence algebras. They play an essential role in \cite{GKT1} and \cite{GKT2} as a natural notion of morphism between decomposition spaces, but the present paper deals also with general simplicial maps.

\revised{
\paragraph{Layered finite posets and layered finite sets}
We refer to \cite{GKT:restrict} for the following material.
An \emph{$n$-layering} of a finite poset $P$ is a monotone map $l: P \to \underline{n}$, where $\underline{n} = \{1, \dots, n \}$ are the objects of the skeleton of the category of finite ordered sets (possibly empty) and monotone maps. The fibres $P_i = l^{-1}(i)$, $i \in \underline{n}$ are called layers, and can be empty.
The objects of the groupoid $\ds{C}_n$ of $n$-layered finite posets are monotone maps $l: P \to \underline{n}$ and the morphisms are triangles
}
\begin{center}
    \begin{tikzcd}[column sep=small]
        P \arrow[rr, "\eq"] 
           \arrow[dr, ""'] & & P' \arrow[dl, ""]\\
        & \underline{n}, & 
    \end{tikzcd}
\end{center}
\revised{
where $P \to P'$ is a monotone bijection.
They assemble into a simplicial groupoid $\ds C$. 
The face maps are given by joining layers, or deleting an outer layer for the top and bottom face maps. The degeneracy maps are given by inserting empty layers. 
\begin{prop}[{\cite[Proposition 6.12, Lemma 6.13]{GKT:restrict}}]
    The simplicial groupoid $\ds C$ of layered finite posets is a decomposition space (but not a Segal space), and is complete, locally finite, locally discrete, and of locally finite length.
\end{prop}
The incidence coalgebra of $\ds{C}$ has comultiplication given by the span
\[
    \ds{C}_1 \xla{d_1} \ds{C}_2 \xra{(d_2, d_0)} \ds{C}_1 \times \ds{C}_1,
\]
where $d_1$ joins the two layers, and $d_2$ and $d_0$ return the two layers.
The comultiplication of a poset is thus obtained by summing over admissible cuts (a $2$-layering of the poset) and taking tensor product of the two layers.

\bigskip

Similarly, let $\mathbf I_n$ denote the groupoid of all layerings of finite sets. Again these groupoids assemble into a simplicial groupoid, denoted $\ds{I}$.

\begin{prop}[{\cite[Proposition 4.3, Lemma 4.4]{GKT:restrict}}]
    The simplicial groupoid $\ds{I}$ is a Segal space, and hence a decomposition space, which is complete, locally finite, locally discrete, and of locally finite length.
\end{prop}

The simplicial groupoid $\ds{C}$ is the decomposition space corresponding to the terminal directed restriction species, finite posets and convex maps, while $\ds{I}$ is the decomposition space corresponding to the terminal restriction species, finite sets and injections.
The incidence coalgebra of $\ds{I}$ is the binomial coalgebra \cite[\S 2.4]{GKT:restrict} with well-known Möbius function $(-1)^n$ for a set with $n$ elements.
}



\section{Bicomodules}
\label{sec:bicomodules}

\subsection{Comodules}
\label{sub:comodules}

The theory of modules in the context of decomposition spaces has been developed by Walde \cite{Walde}, and independently by Young \cite{Young}, both in the context of Hall algebras. They call them relative 2-Segal spaces. Here we give a conceptual way to reformulate their definitions using linear functors.

Given a map between two simplicial $\infty$-groupoids $f: C \to X$, 
the span \\
$C_{0}  \xleftarrow{d_1} C_{1} 
    \xrightarrow{(f_{1},d_0)} X_{1} \times C_{0}$
defines a linear functor
$    \gamma_l: \infGrpd_{/{C_{0}}} \to \infGrpd_{/X_{1}} \otimes \infGrpd_{/{C_{0}}}$, and the span
$C_{0}  \xleftarrow{d_0} C_{1} 
    \xrightarrow{(d_1,f_{1})} C_{0} \times X_{1}$
defines a linear functor
$    \gamma_r: \infGrpd_{/{C_{0}}} \to \infGrpd_{/C_{0}} \otimes \infGrpd_{/{X_{1}}}$.

\begin{prop}\label{comodule}
Let $f: C \to X$ be a map between two simplicial $\infty$-groupoids.
    Suppose moreover that $C$ is Segal, $X$ is a decomposition space and the map $f: C \to X$ is culf, then the span
\[
    C_{0}  \xleftarrow{d_1} C_{1} 
    \xrightarrow{(f_{1},d_0)} X_{1} \times C_{0}
\]
induces on the slice $\infty$-category $\infGrpd_{/C_0}$ the structure of a left $\infGrpd_{/X_1}$-comodule (at the $\pi_0$ level), and the span 
\[
    C_{0}  \xleftarrow{d_0} C_{1} 
    \xrightarrow{(d_1,f_{1})} C_{0} \times X_{1}.
\]
induces on $\infGrpd_{/C_0}$ the structure of a right $\infGrpd_{/X_1}$-comodule (at the $\pi_0$ level).
\end{prop}

The data needed to obtain a comodule is called a \emph{comodule configuration}, that is a culf map from a Segal space to a decomposition space.

\begin{rmk}\label{rmk:comodule}
The relevance of the Segal condition on $C$ and the culf condition on $f$ can be explained individually as follows.
It is standard that for a category $C$, the coalgebra of arrows $C_1$ coacts on $C_0$: the coaction (from the right) is given by
  $b \mapsto \sum_{f:a\to b} a \otimes f$.
Coassociativity of this coaction is equivalent to the Segal condition. 
Now a culf map $C \to X$ defines a coalgebra homomorphism, and in this way, also $X_1$ coacts on $C_0$, by “corestriction of coscalars”.
\end{rmk}

\revised{
\begin{rmk}
    The proposition is stated only at the $\pi_0$-level. This means that we establish only the comodule structure up to homotopy, but do not establish the \emph{coherence} of this up-to-homotopy structure. A stronger result, a partial coherence result, is given by \cite{Walde} and \cite{Young}, who establish the coherence at the $1$-truncated level (rather than the $0$-truncated level established here). It is most likely that full coherence can be established by exploiting the techniques employed by \cite{GKT1} and \cite{Penney}. While only a small bit of the axioms are used to establish the proposition as stated, the full decomposition-space axioms and the culf condition are expected to be required for the fully coherent result, and this is why these conditions have been included in the definition of comodule configuration.
\end{rmk}
}

\begin{proof}
We want to prove that the map $\gamma_l$ is a left $\infGrpd_{/X_{1}}$-coaction.
The desired diagram, commutative up to homotopy 
\begin{center}
  \begin{tikzcd}[sep=large]
    \infGrpd_{/C_{0}} \ar[r, "\gamma_l"] 
                          \ar[d, "\gamma_l"'] 
      & \infGrpd_{/X_{1} \times C_{0}} \ar[d, "\Id_{} \otimes \gamma_l"]\\
    \infGrpd_{/X_{1} \times C_{0}}
        \ar[r, "\Delta \otimes \Id_{}"'] 
      & \infGrpd_{/X_{1} \times X_{1} \times C_{0}}
  \end{tikzcd}
\end{center}
is induced by the solid spans in the diagram
\begin{center}
  \begin{tikzcd}[sep=huge]
    C_{0} 
      & C_{1} \ar[l, "d_1"'] 
                              \ar[r, "{(f_1,d_0)}"]
      & X_{1} \times C_{0} \\
    C_{1} \ar[u, "d_1"] 
                          \ar[d, "{(f_1,d_0)}"'] 
      & C_{2} \ar[l, dashed, "d_1"] 
                              \ar[d, dashed, "{(f_1, d_0 d_0)}"] 
                              \ar[r, dashed, "{(d_2 f_1,d_0)}"'] 
                              \ar[u, dashed, "d_2"]
                              \ar[ur, phantom, "\urcorner", very near start]
                              \ar[dl, phantom, "\llcorner", very near start]
      & X_{1} \times C_{1} \ar[u, "\Id_{} \otimes d_1"'] 
                                          \ar[d, "\Id_{} \otimes {(f_1,d_0)}"]\\
     X_{1} \times C_{0} 
       & X_{2} \times C_{0} \ar[l, "{d_1 \otimes \Id}"]  \ar[r, "{(d_2,d_0) \otimes \Id}"']
       & X_{1} \times X_{1} \times C_{0}.
  \end{tikzcd}
\end{center}
The coassociativity (at the $\pi_0$ level) will follow from Beck-Chevalley equivalences if we have the two pullbacks indicated in the diagram.
The upper right-hand square is a pullback if and only if its composite with the second projection is a pullback.
This composite outer square is a pullback because $C$ satisfies the Segal condition.
Similarly, the lower left-hand square is a pullback if its composite with the first projection is a pullback.
This composite outer square is a pullback because $f: C \to X$ is culf.
\end{proof}


\revised{
\begin{ex}[Décalage \cite{Illusie}]\label{exDec}
Given a simplicial space $X$, the \emph{lower décalage} $\Dec_\bot (X)$ is the simplicial space obtained by deleting $X_0$, all $d_0$ face maps and $s_0$ degeneracy maps. The original $d_0$ maps induce a simplicial map $d_\bot : \Dec_\bot (X) \to X$, called the décalage map.
Similarly, the \emph{upper décalage} $\Dec_\top (X)$ is the simplicial space obtained by deleting $X_0$, all last face maps $d_\top$ and last degeneracy maps $s_\top$. The original $d_\top$ maps induce a simplicial map $d_\top : \Dec_\top X \to X$.

It is well known that $\Dec_\bot(X)$ is a Segal space and the décalage map 
is culf (see \cite[Proposition 4.9]{GKT1}). Hence we have a comodule configuration. The resulting comodule is the incidence coalgebra of $X$ as a (right) comodule over itself.
\end{ex}
}

\medskip

For categories, given a functor $f: C \to D$, define the \emph{mapping cylinder} (or \emph{collage} in \cite{JoyalCL}) $M_f$ to be the category where objects are either objects of $C$ or objects of $D$ and
\[
    \Hom_{M_f}(x,y) =
        \begin{cases}
             \Hom_{C}(x,y)    \, &\text{ if } x, y \in C,\\
             \Hom_{D}(x,y)    \, &\text{ if } x, y \in D,\\
             \Hom_{D}(f(x),y) \, &\text{ if } x \in C, y \in D,\\
             \emptyset        \, &\text{ else }.\\
        \end{cases}
\]
There exists a unique $p: M_f \to \Delta^1$ such that $p^{-1}(0) = C$ and $p^{-1}(1) = D$. This is moreover a cocartesian fibration, the cocartesian lift for $x \in C$ being given by $\Id_{f(x)} \in \Map_{M_f}(x,f(x))$.
The shape of a comodule configuration is that of $(M_{\id})\op$, where $M_{\id}$ is the mapping cylinder of the identity of $\simplexcat$.
\revised{
In other words, a comodule configuration is a functor from $(M_{\id})\op$ to $\infGrpd$ (satisfying certain conditions).}

Let $\simplexcat_{\text{bot}}$ be the simplex category of finite linear orders with a specified bottom element, and bottom-preserving monotone maps.
Consider the mapping cylinder $M_j$ of the functor $j: \simplexcat \to \simplexcat_{\text{bot}}$ freely adding a bottom element.
Presheaves on $M_j$ are diagrams of the following shape.
\begin{center}
    \begin{tikzcd}[sep=large]
         & X_{0} \ar[r, "s_0" description, near end] 
         & X_{1} \ar[l, shift left=1.5, "d_0", near end]
                    \ar[l, shift right=1.5, "d_1"', near end]
                    \ar[r, shift left=1.5, "s_1" description, near end] 
                    \ar[r, shift right=1.5, "s_0" description, near end]
         & X_{2} \ar[l, "d_1" description, near end]
                    \ar[l, shift left=3, "d_0", near end] 
                    \ar[l, shift right=3, "d_2"', near end] 
                    \ar[r, phantom, "\dots"]
                    & \phantom{} \\
       C_{-1} \ar[r, bend right]
         & C_{0} \ar[r, bend right,shift right=1]
                   \ar[u, "v"']
                   \ar[l, "u"']
                   \ar[r, "s_0" description] 
         & C_{1} \ar[r, bend right,shift right=2.5]
                   \ar[u, "v"']
                   \ar[l, shift left=1.5, "d_0", near end]
                    \ar[l, shift right=1.5, "d_1"', near end]
                   \ar[r, shift left=1.5, "s_1" description, near end] 
                   \ar[r, shift right=1.5, "s_0" description, near end] 
         & C_{2} \ar[u, "v"']
                   \ar[l, "d_1" description, near end]
                   \ar[l, shift left=3, "d_0", near end] 
                   \ar[l, shift right=3, "d_2"', near end]
                   \ar[r, phantom, "\dots"]
                    & \phantom{}
    \end{tikzcd}
\end{center}

This is the shape of what we call a \emph{right pointed comodule configuration}: it is a comodule configuration $C \to X$ such that the Segal space $C$ is augmented, and with new bottom sections $s_{-1}: C_{n-1} \to C_n$.
The importance of the pointing (the extra bottom degeneracy maps) is that it makes possible to formulate the notion of completeness and the condition locally finite length, see Section~\ref{sub:Mobbicomod} below; it guarantees the existence of a filtration on the associated comodule (see \cite[\S 6]{GKT2} for a similar argument), which is of independent interest.


\revised{
\begin{ex}
    The comodule configuration obtained from the lower décalage of a decomposition space $X$ is also right pointed, the augmentation map is given by $d_1: X_1 \to X_0$, and the extra bottom sections by $s_0$.
\end{ex}
}

\subsection{Augmented bisimplicial infinity-groupoids}
\label{sub:bisimplicialgrpd}

We shall establish conditions under which left and right comodule structures define a bicomodule.
The main objects of interest are augmented bisimplicial $\infty$-groupoids subject to conditions, which are formulated in terms of pullbacks.
We consider the functor $\infty$-category
\[
    \Fun{(\simplexcat\op \times \simplexcat\op, \infGrpd)}
\]
whose objects are \revised{\emph{bisimplicial $\infty$-groupoids}, that is} functors from the $\infty$-category $\simplexcat\op \times \simplexcat\op$ to the $\infty$-category $\infGrpd$.


A \emph{double Segal space} is a bisimplicial $\infty$-groupoid satisfying the Segal condition for each restriction $\simplexcat\op \times \{[n]\} \to \infGrpd$ (the columns) and $\{[n]\} \times \simplexcat\op \to \infGrpd$ (the rows).

\revised{
Let $\simplexcat_{+}$ be the augmented simplex category of all finite ordinals and order-preserving maps.
}
An \emph{augmented} bisimplicial $\infty$-groupoid $B$ has in addition $\infty$-groupoids $B_{i,-1}$ and $B_{-1,i}$ of $(-1)$-simplices.
We consider the functor $\infty$-category
\[
    \Fun{(\simplexcat\op_{+} \times \simplexcat\op_{+} \backslash \{(-1,-1)\}, \infGrpd)}
\]
whose objects are \emph{augmented bisimplicial $\infty$-groupoids}.

\begin{rmk}\label{shape1}
    The shape of an augmented bisimplicial $\infty$-groupoid is $(\simplexcat_{/ \Delta^1})\op$.
    We denote $[i,j]$ the object given by the map $\Delta^{i+1+j} \to \Delta^1$ sending the $i+1$ first vertices to $0$ and the others to $1$.
We allow $i$ or $j$ to be equal to $-1$ but not both.
Maps $[i,j] \to [k,l]$ are given by the inclusions respecting the horizontal map.
For example, the object $[2,1]$ can be drawn as follows
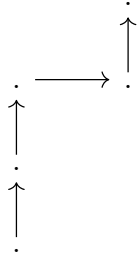
\begin{center}
  \begin{tikzcd}\label{picturejump}
             & .        \\
    . \ar[r] & . \ar{u} \\
    . \ar[u] &          \\
    . \ar[u] &  
  \end{tikzcd}
\end{center}
where the horizontal maps lie over the map in $\Delta^1$.

    We can draw $[i,j]$ as a column of $i{+}1$ black dots followed by $j{+}1$ white dots. Maps send black dots to black dots and white dots to white dots, without crossing.

\end{rmk}

We use the following notation for an augmented bisimplicial $\infty$-groupoid. We denote $d_k: B_{i,j} \to B_{i,j-1}$ and $e_l: B_{i,j} \to B_{i-1,j}$ the face maps, and $s_k: B_{i,j-1} \to B_{i,j}$ and $t_l: B_{i-1,j} \to B_{i,j}$ the degeneracy maps; $u$ and $v$ are the augmentation maps.

\begin{center}
    \begin{tikzcd}[sep=large]
         & B_{-1,0} \ar[r, "s_0" description, near end] 
         & B_{-1,1} \ar[l, shift left=1.5, "d_0", near end]
                    \ar[l, shift right=1.5, "d_1"', near end]
                    \ar[r, shift left=1.5, "s_1" description, near end] 
                    \ar[r, shift right=1.5, "s_0" description, near end]
         & B_{-1,2} \ar[l, "d_1" description, near end]
                    \ar[l, shift left=3, "d_0", near end] 
                    \ar[l, shift right=3, "d_2"', near end] 
                    \ar[r, phantom, "\dots"]
                    & \phantom{} \\
       B_{0,-1} \ar[d, "t_0" description]
         & B_{0,0} \ar[d, "t_0" description]
                   \ar[u, "v"']
                   \ar[l, "u"']
                   \ar[r, "s_0" description] 
         & B_{0,1} \ar[d, "t_0" description]
                   \ar[u, "v"']
                   \ar[l, shift left=1.5, "d_0", near end]
                    \ar[l, shift right=1.5, "d_1"', near end]
                   \ar[r, shift left=1.5, "s_1" description, near end] 
                   \ar[r, shift right=1.5, "s_0" description, near end] 
         & B_{0,2} \ar[u, "v"']
                   \ar[d, "t_0" description]
                   \ar[l, "d_1" description, near end]
                   \ar[l, shift left=3, "d_0", near end] 
                   \ar[l, shift right=3, "d_2"', near end]
                   \ar[r, phantom, "\dots"]
                    & \phantom{}\\
       B_{1,-1} \ar[u, shift left=1.5, "e_0", near end]
                \ar[u, shift right=1.5, "e_1"', near end]
                \ar[d, phantom, "\vdots"]
         & B_{1,0} \ar[u, shift left=1.5, "e_0", near end]
                   \ar[u, shift right=1.5, "e_1"', near end]
                   \ar[l, "u"'] 
                   \ar[r, "s_0" description, near end] 
                    \ar[d, phantom, "\vdots"]
         & B_{1,1} \ar[u, shift left=1.5, "e_0", near end]
                   \ar[u, shift right=1.5, "e_1"', near end]
                   \ar[l, shift left=1.5, "d_0", near end]
                    \ar[l, shift right=1.5, "d_1"', near end]
                   \ar[r, shift left=1.5, "s_1" description, near end] 
                   \ar[r, shift right=1.5, "s_0" description, near end] 
                    \ar[d, phantom, "\vdots"]
         & B_{1,2}   
                   \ar[u, shift left=1.5, "e_0", near end]
                   \ar[u, shift right=1.5, "e_1"', near end]
                   \ar[l, "d_1" description, near end]
                   \ar[l, shift left=3, "d_0", near end] 
                   \ar[l, shift right=3, "d_2"', near end]
                   \ar[r, phantom, "\dots"]
                    \ar[d, phantom, "\vdots"]
                    & \phantom{}\\
       \phantom{}
          & \phantom{}
          &\phantom{}
          & \phantom{}
    \end{tikzcd}
\end{center}

An augmented double Segal space satisfies that rows and columns are Segal. If we suppose that augmentations are culf and $B_{\bullet,-1}$ and $B_{-1,\bullet}$ are decomposition spaces, we can apply Proposition \ref{comodule} to obtain comodules: the span
\[B_{0,0}  \xleftarrow{e_1} B_{1,0} 
    \xrightarrow{(u,e_0)} B_{1,-1} \times B_{0,0}\]
    induces on $\infGrpd_{/B_{0,0}}$ the structure of a left comodule over $\infGrpd_{/B_{1,-1}}$, and the span
\[B_{0,0}  \xleftarrow{d_0} B_{0,1} 
    \xrightarrow{(d_1,v)} B_{0,0} \times B_{-1,1}\]
induces on $\infGrpd_{/B_{0,0}}$ the structure of a right comodule over $\infGrpd_{/B_{-1,1}}$.


\subsection{Stability}
\label{sub:stability}

We say a bisimplicial $\infty$-groupoid is \emph{stable} if the following squares are pullbacks:

\begin{center}
  \begin{tikzcd}
    B_{i-1,j-1} 
        & B_{i-1,j}  \arrow[l, "d_k"'] \\
    B_{i,j-1} \arrow[u, "e_l"]  
        & B_{i,j} \arrow[l, "d_k"] 
                  \arrow[u, "e_l"'] 
                  \arrow[ul, phantom, "\ulcorner", very near start],
  \end{tikzcd}
  \qquad
  \begin{tikzcd}
    B_{i-1,j-1} \arrow[r, "s_k"]
        & B_{i-1,j}   \\
    B_{i,j-1} \arrow[u, "e_l"] \arrow[r, "s_k"'] \arrow[ur, phantom, "\urcorner", very near start]
        & B_{i,j} 
                  \arrow[u, "e_l"'] ,
  \end{tikzcd}
  
  \end{center} 
  \begin{center}
    \begin{tikzcd}
    B_{i-1,j-1}  \arrow[d, "t_l"]
        & B_{i-1,j} \arrow[d, "t_l"] \arrow[l, "d_k"'] \arrow[dl, phantom, "\llcorner", very near start]  \\
    B_{i,j-1} 
        & B_{i,j} \arrow[l, "d_k"],
  \end{tikzcd}
    \qquad
    \begin{tikzcd}
    B_{i-1,j-1} \arrow[r, "s_k"] \arrow[d, "t_l"] \arrow[dr, phantom, "\lrcorner", very near start]
        & B_{i-1,j} \arrow[d, "t_l"]   \\
    B_{i,j-1}  \arrow[r, "s_k"'] 
        & B_{i,j},
  \end{tikzcd}
\end{center}
for all values of the indices except for $d_\bot$ along $e_\top$ and $d_\top$ along $e_\bot$.

\begin{rmk}
A bisimplicial $\infty$-groupoid is stable if it satisfies all the following properties:
    \begin{itemize}
        \item $s_k: B_{i,j-1} \to B_{i,j}$ is a cartesian natural transformation, for all $0 \le k \le j-1$;
        \item $d_k$, $k \neq \top, \bot$, is a cartesian natural transformation;
        \item $d_\top$ is a left fibration;
        \item $d_\bot$ is a right fibration.
    \end{itemize}
\end{rmk}

\begin{rmk}
Bergner, Osorno, Ozornova, Rovelli, and Scheimbauer introduced the notion of stable double category (bisimplicial {\em set}) in \cite{BOORS}: they define a double category to be stable if every square is uniquely determined by its span of source morphisms and, independently by its cospan of target morphisms.
The present definition is a categorical reformulation of their notion suitable for $\infty$-groupoids.
The motivation for the terminology is the following example.
Let $C$ be a stable $\infty$-category (see \cite[Chapter 1]{LurieHA}).
Define a double Segal space $B$ where $B_{0,0}$ is the $\infty$-groupoid of objects of $C$, where $B_{0,1}$ is the $\infty$-groupoid of arrows of $C$ (as in the fat nerve), and $B_{1,1}$ is the $\infty$-groupoid of pullback squares (equivalently, pushout squares).
More generally, $B_{m,n}$ is the $\infty$-groupoid of ($\Delta^m {\times}\Delta^n$)-diagrams in $C$ for which all the rectangles are pullbacks (and hence pushouts).
This is a stable bisimplicial $\infty$-groupoid (which of course is a double Segal space).
This is almost by definition: since we only took pullback and pushout squares, they are determined by their sources by pushout or their targets by pullback, in the sense of our definition.
\end{rmk}

\begin{lem}\label{stablelemma}
Let $B$ be a double Segal space.
Suppose we have the two following pullbacks:
\begin{center}
  \begin{tikzcd}
    B_{0,0} 
        & B_{0,1}  \arrow[l, "d_0"'] \\
    B_{1,0} \arrow[u, "e_0"]  {stablelemma}
        & B_{1,1} \arrow[l, "d_0"] 
                  \arrow[u, "e_0"'] 
                  \arrow[ul, phantom, "\ulcorner", very near start]
  \end{tikzcd}
  \qquad
  \begin{tikzcd}
    B_{0,0} 
        & B_{0,1}  \arrow[l, "d_1"'] \\
    B_{1,0} \arrow[u, "e_1"]  
        & B_{1,1}, \arrow[l, "d_1"] 
                  \arrow[u, "e_1"'] 
                  \arrow[ul, phantom, "\ulcorner", very near start]
  \end{tikzcd}
\end{center}
then the double Segal space is stable.
\end{lem}

\begin{proof}
    First, the second pullback implies that every square with top face maps \revised{$d_\top : B_{i,j+1} \to B_{i,j}$} is a pullback.
Indeed, in the cube
\begin{center}
  \begin{tikzcd}
      & B_{0,0}  & & B_{0,1} \arrow[ll, "d_{\top}"']\\
     B_{0,1} \ar[ur, "d_\bot"] & & B_{0,2}  \ar[ur, "d_\bot"] &\\
      & B_{1,0} \arrow[uu, "e_\top"',near start] \arrow[from=ur, to=ul, crossing over, "d_{\top}"',very near start] & & B_{1,1}, \arrow[uu, "e_\top"'] \arrow[ll, "d_{\top}"',very near start] \\
     B_{1,1} \arrow[uu, "e_\top",near start] \ar[ur, "d_\bot"]& & B_{1,2} \arrow[uu, crossing over, "e_\top"',near start] \arrow[ll, "d_{\top}"] \ar[ur, "d_\bot"']&
  \end{tikzcd}
\end{center}
the top and bottom squares are pullbacks because every row is Segal, and the back square is a pullback by hypothesis. Thus the rectangle consisting of bottom and back is a pullback because bottom and back squares are; front is a pullback because top and rectangle are.
By induction, suppose the squares
\begin{center}
  \begin{tikzcd}
    B_{i-1,j-1} 
        & B_{i-1,j}  \arrow[l, "d_\top"'] \\
    B_{i,j-1} \arrow[u, "e_\top"]  
        & B_{i,j} \arrow[l, "d_\top"] 
                  \arrow[u, "e_\top"'] 
                  \arrow[ul, phantom, "\ulcorner", very near start]
  \end{tikzcd}
\end{center}
are pullback, we can form cubes with the top and bottom faces pullbacks thanks to the Segal condition, and the back square is a pullback by hypothesis.
This proves that every square involving top face maps are pullbacks. Starting with the first pullback, we prove in the same way that every square involving bottom face maps are pullbacks.

Now we want to prove that the following square is a pullback, for $0 < k < i$,
\begin{center}
  \begin{tikzcd}
    B_{i-1,j-1}
        & B_{i-1,j}  \arrow[l, "d_\top"'] \\
    B_{i,j-1} \arrow[u, "e_k"]  
        & B_{i,j} \arrow[l, "d_\top"] 
                  \arrow[u, "e_k"'].
  \end{tikzcd}
\end{center}
\revised{
If $k=i-r$, we postcompose vertically with $(e_\top)^r$: 
}
\begin{center}
  \begin{tikzcd}
     B_{k-1,j-1} & B_{k-1,j} \ar[l, "d_\top"']
     \\
     B_{i-1,j-1} \ar[u, "(e_\top)^r"] & B_{i-1,j} \ar[u, "(e_\top)^r"'] \ar[l, "d_\top"]
     \\
     B_{i,j-1} \ar[u, "e_k"] & B_{i,j} \ar[u, "e_k"'] \ar[l, "d_\top"]
  \end{tikzcd} 
\end{center}
\revised{
Then the vertical composite is equivalent to $e_\top \circ (e_\top)^r$ (by face map identities), so both the rectangle and the upper square are  pullbacks by assumption, and therefore by Lemma~\ref{prismlemma}, the lower square is a pullback too, as required. 
}



We can do the same proof with bottom face maps. We can also replace $d_\top$ in the new previous pullback squares and obtain the remaining pullbacks involving the face maps.

For squares with face and degeneracy maps, we use the following strategy: in the diagram

\begin{center}
    \begin{tikzcd}
        B_{00} \arrow[r, "s_0"] & B_{01} \arrow[r, "d_\bot"] & B_{00}\\
        B_{10} \arrow[r, "s_0"'] \arrow[u, "e_{\bot}"] & B_{11} \arrow[r, "d_\bot"'] \arrow[u, "e_{\bot}"'] & B_{10} \arrow[u, "e_{\bot}"'],
    \end{tikzcd}
\end{center}
the map $s_0$ is a section of $d_1$, then the long edge is an identity. The right-hand square is a pullback (it is one of the two pullback in the hypothesis). Thus the left-hand square is a pullback.
We can proceed in the same way for the other degeneracy maps.

There remains the case of squares involving only degeneracy:
\begin{center}
\begin{tikzcd}
    B_{ij} \arrow[r, "s_0"] 
       \arrow[d, "t_k"'] & B_{i,j+1} \arrow[d, "t_k"]\\
    B_{i+1, j} \arrow[r, "s_0"'] & B_{i+1, j+1}.
\end{tikzcd}
\end{center}
We again glue on the right a square with face map such that the long edge is an identity and use once again the Lemma \ref{prismlemma}.
\end{proof}


\subsection{Bicomodules}
\label{sub:bicomodules}

\begin{thm}\label{thm:bicomodule}
Let $B$ be an augmented stable double Segal space, and such that the augmentation maps are culf.
Suppose moreover $X := B_{\bullet,-1}$ and $Y:= B_{-1,\bullet}$ are decomposition spaces.
Then the spans
 \[B_{0,0}  \xleftarrow{e_1} B_{1,0} 
    \xrightarrow{(u,e_0)} X_{1} \times B_{0,0}
\]
and
\[B_{0,0}  \xleftarrow{d_0} B_{0,1} 
    \xrightarrow{(d_1,v)} B_{0,0} \times Y_{1}
\]
induce on $\infGrpd_{/B_{0,0}}$ the structure of a bicomodule over $\infGrpd_{/X_1}$ and $\infGrpd_{/Y_1}$  (at the $\pi_0$ level).
\end{thm}

A bisimplicial $\infty$-groupoid satisfying the conditions of the theorem is called a \emph{bicomodule configuration}.

\revised{
\begin{rmk}
  In analogy with Remark~\ref{rmk:comodule}, the notion of bicomodule configuration can be broken up into steps. First, for any double Segal space $B$, since the zeroth column $B_{\bullet,0}$ is a Segal space, $\infGrpd_{/B_{0,0}}$ is a left comodule over $\infGrpd_{/B_{1,0}}$, and similarly $\infGrpd_{/B_{0,0}}$ is a right comodule over $\infGrpd_{/B_{0,1}}$.  It is now the stability of $B$ that expresses the bicomodule condition.  From here, a culf augmentation on the left to a decomposition space $X$ induces a coalgebra homomorphism, and a culf augmentation on the right to a decomposition space $Y$ induces another coalgebra homomorphism, and coextension of coscalars along these coalgebra homomorphisms makes $\infGrpd_{/B_{0,0}}$ an $\infGrpd_{/X_1}$-$\infGrpd_{/Y_1}$ bicomodule.  This viewpoint might well be useful in the proof of full coherence.
\end{rmk}

\begin{rmk}
    As for the Proposition~\ref{comodule}, the theorem is stated at the $\pi_0$-level. It is most likely that the full coherence can be established using the techniques employed in \cite{GKT1} and \cite{Penney}. It is expected that all the stability pullbacks are required for the fully coherent result.
For the present purposes, we are going to take homotopy cardinality anyway, and for that, coherence is not essential.
\end{rmk}
}

\begin{proof}
The left and right comodule structures were established in Proposition~\ref{comodule}.
The desired homotopy coherent diagram 
\begin{center}
  \begin{tikzcd}[sep=huge]
    \infGrpd_{/B_{0,0}} \ar[r, "\gamma_l"] \ar[d, "\gamma_r"'] 
      & \infGrpd_{/B_{1,-1} \times B_{0,0}} \ar[d, "\Id_{} \otimes \gamma_r"] \\
    \infGrpd_{/B_{0,0} \times B_{-1,1}} \ar[r, "\gamma_l \otimes \Id_{}"'] 
      & \infGrpd_{/B_{1,-1} \times B_{0,0} \times B_{-1,1}}
  \end{tikzcd}
\end{center}
is induced by the solid spans in the diagram

\begin{center}
  \begin{tikzcd}[sep=huge]
    B_{0,0} 
      & B_{1,0} \ar[l, "e_1"'] 
                              \ar[r, "{(u,e_0)}"]
      & B_{1,-1} \times B_{0,0}\\
    B_{0,1} \ar[u, "d_0"] 
                          \ar[d, "{(d_1,v)}"'] 
      & B_{1,1} \ar[l, dashed, "e_1"] 
                              \ar[d, dashed, "{(d_1,v e_1)}"] 
                              \ar[r, dashed, "{(u d_0,e_0)}"'] 
                              \ar[u, dashed, "d_0"]
                              \ar[ur, phantom, "\urcorner", very near start]
                              \ar[dl, phantom, "\llcorner", very near start]
      & B_{1,-1} \times B_{0,1} \ar[u, "\Id \otimes d_0"'] 
                                          \ar[d, "\Id \otimes {(d_1,v)}"]\\
      B_{0,0} \times B_{-1,1}
       &  B_{1,0} \times B_{-1,1} \ar[l, "{e_1 \otimes \Id }"]  \ar[r, "{{(u, e_0)} \otimes \Id}"']
       & B_{1,-1} \times B_{0,0} \times B_{-1,1}.
  \end{tikzcd}
\end{center}
The homotopy commutativity of the squares follows one again from the new augmentation simplicial identities.
The upper-right hand square is a pullback if and only if its composite with the second projection is a pullback and, similarly, the lower-left hand square is a pullback if and only if its composite with the first projection is a pullback.
These composite outer squares are pullbacks due to the stability condition.
\end{proof}

\revised{
\begin{ex}
In analogy with Example~\ref{exDec}, given a decomposition space $X$, let $B$ be the total decalage of $X$. (Its zeroth column is $\Dec_\top(X)$ and its zeroth row is $\Dec_\bot(X)$. With its natural augmentation maps, this becomes a bicomodule configuration, realising the coalgebra of $X$ as a bicomodule over itself.
\end{ex}
}


\revised{
\subsection{Augmented double Segal space of layered sets and posets}\label{sec:layeredsetsposets}

The following example is developped in \cite{Ca:mdrs}, where details of the general construction and proofs can be found.

Suppose $\ds{I}$ is the decomposition space of layered finite sets and $\ds{C}$ is the decomposition space of layered finite posets.
We obtain a bisimplicial groupoid $\ds{B}$ with the following explicit description: the groupoid $\ds{B}_{i,j}$ consists of pairs of layerings $(S \to \underline{i},P \to \underline{j{+}1})$ where $S$ is a finite set, and $P$ is a finite poset.
For example, $\ds{B}_{0,0}$ is the groupoid of $1$-layered finite posets.
The horizontal face maps (taking place only on the $(j{+}1)$-layered finite poset part) are given by:
\begin{itemize}
    \item $d_k:\ds{B}_{i,j} \to \ds{B}_{i,j-1}$ joins the layers $(k{+}1)$ and $(k{+}2)$ of the poset, for all $j >0$ and $0 \le k \le j-1$;
    \item  $d_{\top} = d_{j}: \ds{B}_{i,j} \to \ds{B}_{i,j-1}$ deletes the last layer.
\end{itemize}
Horizontal degeneracy maps are given by inserting empty layers:
$s_k: \ds{B}_{i,j} \to \ds{B}_{i, j+1}$ inserts an empty $(k{+}2)$nd layer in the poset, for all $j\ge 0$ and $0\le k \le j$.

The vertical face maps are given by:
\begin{itemize}
    \item $e_\bot = e_0: \ds{B}_{i,j} \to \ds{B}_{i-1, j}$ deletes the first layer of the set, for all $i > 0$;
    \item $e_k: \ds{B}_{i,j} \to \ds{B}_{i-1, j}$ joins the layers $k$ and $k{+}1$ of the set, for all $0 < k < i$.
    \item $e_\top = e_i: \ds{B}_{i,j} \to \ds{B}_{i-1, j}$ joins the last layer of the set and the first layer of the poset into the first layer of the poset.
\end{itemize}
Vertical degeneracy maps are given by inserting empty layers:  $t_k: \ds{B}_{i,j} \to \ds{B}_{i+1, j}$ inserts an empty $(k{+}1)$st layer to the set, for all $0 \le k \le i$.
The augmentation maps are  
$u: \ds{B}_{i,0} \to \ds{I}_i$ deleting the whole $1$-layered poset and
$v: \ds{B}_{0, j} \to \ds{C}_j$ deleting the first layer of the poset.
It should be noted that the row $\ds{B}_{0,\bullet}$ is the lower décalage of $\ds{C}$, that $v$ is the décalage map given by the original $d_0$, and that $u$ is the augmentation map that décalage always have.

\begin{prop}[{\cite[Proposition~2.14]{Ca:mdrs}}]
    With augmentations maps $u$ and $v$, the bisimplicial groupoid $\ds{B}$ is a bicomodule configuration.
\end{prop}
 }

\section{Correspondences, fibrations, and adjunctions}
\label{sec:adj}

\subsection{Decomposition space correspondences}
\label{sub:corresp}


A \emph{correspondence} is by definition a decomposition space $\mathcal{M}$ with a map to the \revised{$1$-simplex $\Delta^1$}.
We consider the slice $\infty$-category $\Cat_{\infty/\Delta^1}$. It contains in particular $\simplexcat_{/\Delta^1}$, whose objects are $[i,j]$, see Remark~\ref{shape1}.
There is now a natural notion of nerve in this context.
Given a correspondence \revised{$p: \mathcal{M} \to \Delta^1$}, the relative nerve 
$N_{\Delta^1}: \Cat_{\infty/\Delta^1} \to \Fun((\simplexcat_{/\Delta^1})\op, \infGrpd)$ of $p$ is the augmented bisimplicial $\infty$-groupoid given by 
$B_{i,j} := N_{\Delta^1}(p)_{i,j} = \Map_{/\Delta^1}([i,j], p)$, where $[i,j]$ is given by the map $\Delta^{i+1+j} \to \Delta^{1}$ sending the $i+1$ first vertices to $0$ and the others to $1$. It is allowed for $i$ or $j$ to be equal to $-1$ but not both.

From the nerve definition, the following square is a standard mapping-space fibre sequence for slices:
\begin{center}
\begin{tikzcd}\label{Bij}
    B_{i,j} \arrow[r, ""] 
       \arrow[d, ""'] \arrow[dr, phantom, "\lrcorner", very near start] & \Map(\Delta^{i+1+j}, \mathcal{M}) \arrow[d, "\text{post }p"]\\
    1 \arrow[r, "\name{[i,j]}"'] & \Map(\Delta^{i+1+j}, \Delta^1).
\end{tikzcd}
\end{center}




\begin{prop}\label{runex}
    Given a decomposition space correspondence $p: \mathcal{M} \to \Delta^1$, the bisimplicial $\infty$-groupoid $B$ described above enjoys the following properties:
\begin{enumerate}
    \item it is Segal in both directions;
    \item it is stable;
    \item it is augmented; 
    \item these augmentations are culf.
\end{enumerate}
\end{prop}

To prove these properties, we will use the following lemmas. 

\begin{lem}\label{pbkfibresq}
    Given a diagram such that top and bottom are two fibre sequences
    \begin{center}
        \begin{tikzcd}
            F \arrow[r, ""] 
               \arrow[d, ""'] & E \arrow[r, ""] \arrow[d, ""] & B 
               \arrow[d, "q"]\\
            F' \arrow[r, ""'] & E' \arrow[r, ""] & B'.
        \end{tikzcd}
    \end{center}
    If $q$ is an equivalence, then the left-hand square is a pullback.
\end{lem}

\begin{proof}
\revised{
In the following cube, the front and back squares are pullbacks by assumption; the bottom one is since $q$ is an equivalence.
}
      \begin{center}
    \begin{tikzcd}
        F \arrow[rr, ""] \arrow[dr, ""'] \arrow[dd, ""'] & & E  \arrow[dd, ""] \arrow[dr, ""'] \arrow[dr, ""]\\
        & F' \arrow[rr, crossing over, ""] && E' \arrow[dd, ""] \\
           1 \arrow[rr, ""] \arrow[dr, ""'] && B \arrow[dr, "q"'] \\
           & 1 \arrow[from=uu, crossing over, ""'] \arrow[rr, crossing over, ""]  && B'.
    \end{tikzcd}
  \end{center}
\revised{
The rectangle consisting of the back square and the bottom square is a pullback by Lemma~\ref{prismlemma}, since both squares are pullbacks. Thus the rectangle consisting of the top square and the front one is. Applying one more time Lemma~\ref{prismlemma}, we conclude the top square is a pullback.
}
\end{proof}

\begin{lem}\label{pbksqfibresq}
    Given a diagram such that horizontal maps form fibre sequences
  \begin{center}
    \begin{tikzcd}[column sep=small]
        F_4 \arrow[rr, ""] \arrow[dr, ""'] \arrow[dd, ""'] & & E_4 \arrow[rr, ""] \arrow[dd, ""] \arrow[dr, ""'] & & B_4 \arrow[dd, "u" near end] \arrow[dr, ""]\\
        & F_3 \arrow[rr, crossing over, ""] && E_3 \arrow[rr, crossing over, ""] && B_3 \arrow[dd, "v" near end]\\
           F_2 \arrow[rr, ""] \arrow[dr, ""'] && E_2 \arrow[dr, ""'] \arrow[rr, ""] & & B_2 \arrow[dr, ""]\\
           & F_1 \arrow[from=uu, crossing over, ""'] \arrow[rr, crossing over, ""]  && E_1 \arrow[from=uu, crossing over, ""'] \arrow[rr, ""] & & B_1.
    \end{tikzcd}
  \end{center}
    Suppose the vertical middle square (involving $E_i, 1\le i\le 4$) is a pullback, and
    suppose $u$ and $v$ are equivalences, then the left vertical square is a pullback.
\end{lem}

\revised{
\begin{proof}
    By Lemma~\ref{pbkfibresq}, since $u$ and $v$ are equivalences, the front and back squares of the left cube of the diagram are pullbacks. We conclude by applying Lemma~\ref{prismlemma} twice, as in the proof of Lemma~\ref{pbkfibresq}.
\end{proof}
}

\begin{proof}[Proof of Proposition~\ref{runex}]
(1) Segal in both directions means: for any $i$, the squares
\begin{center}
    \begin{tikzcd}
        B_{n+1,i} \arrow[dr, phantom, "\lrcorner", very near start] \arrow[r, "e_0"] 
           \arrow[d, "e_{n+1}"'] & B_{n,i} \arrow[d, "e_n"]\\
         B_{n,i} \arrow[r, "e_0"'] & B_{n-1,i},
    \end{tikzcd}
    \quad
    \begin{tikzcd}
        B_{i,n+1} \arrow[dr, phantom, "\lrcorner", very near start] \arrow[r, "d_0"] 
           \arrow[d, "d_{n+1}"'] & B_{i,n} \arrow[d, "d_n"]\\
         B_{i,n} \arrow[r, "d_0"'] & B_{i,n-1}.
    \end{tikzcd}
\end{center}
are pullbacks.

The $\infty$-groupoids $B_{n,i}$, for $n \ge 0$ are also given by the following fibre sequences:
\begin{center}
\begin{tikzcd}
    B_{n,i} \arrow[r, ""] 
       \arrow[d, ""'] \arrow[dr, phantom, "\lrcorner", very near start] & \Map(\Delta^{n+1+i}, \mathcal{M}) \arrow[d, "R_{n+1+i}"]\\
    1 \arrow[r, "\name{\id}"'] & \Map(\Delta^1, \Delta^1),
\end{tikzcd}
\end{center}
the right-hand map $R_{n+1+j}$ sends $\sigma \in \Map(\Delta^{n+1+j}, \mathcal{M})$ to $p \circ \sigma \circ \rho_{n+1}$ where 
$\rho_{n+1} : \Delta^1 \to \Delta^{n+1+j}$ maps the arrow in $\Delta^1$ to the $(n+1)$st edge of $\Delta^{n+1+j}$, that is $\rho_{n+1} = (d^\bot)^{n}(d^\top)^{j}$.
Indeed, in the diagram
\begin{center}
\begin{tikzcd}
    B_{i,j} \arrow[r, ""] 
       \arrow[d, ""'] \arrow[dr, phantom, "\lrcorner", very near start] & \Map(\Delta^{i+1+j}, \mathcal{M}) \arrow[d, "\text{post }p"]\\
           1 \arrow[r, "\name{\beta_{i,j}}"'] \arrow[d, ""']& \Map(\Delta^{i+1+j}, \Delta^{1}) \arrow[d, "\text{pre } \rho_{i+1}"]\\
    1 \arrow[r, "\name{\id}"'] & \Map(\Delta^1, \Delta^1),
\end{tikzcd}
\end{center}
the bottom square is a pullback, because the fibre of the right bottom map is contractible, thus the whole rectangle is a pullback by Lemma~\ref{prismlemma}.

Using the Lemma~\ref{pbkfibresq}, we only have to check that the front square in the cube 
\begin{center}
\begin{tikzcd}[column sep=tiny]
    B_{n+1,i} \arrow[rr, ""] \arrow[dr, ""'] 
       \arrow[dd, ""'] & & B_{n+1,i} \arrow[dd, ""] \arrow[dr, ""'] &\\
       & \Map(\Delta^{n+1+1+i}, \mathcal{M}) \arrow[rr, crossing over, ""] && \Map(\Delta^{n+1+i}, \mathcal{M}) \arrow[dd, ""'] \\
       B_{n,i} \arrow[rr, ""] \arrow[dr, ""'] && B_{n-1,i} \arrow[dr, ""'] & \\
       & \Map(\Delta^{n+1+i}, \mathcal{M}) \arrow[from=uu, crossing over, ""'] \arrow[rr, ""] && \Map(\Delta^{n+i}, \mathcal{M})
\end{tikzcd}
\end{center}
is a pullback, and apply Lemma~\ref{prismlemma}.
But the squares
\begin{center}
    \begin{tikzcd}
        \mathcal{M}_{n+1+1+i} \arrow[r, "d_{\bot}"] \arrow[d, "d_{n+1}"'] 
            & \mathcal{M}_{n+1+i} \arrow[d, "d_n"]\\
         \mathcal{M}_{n+1+i} \arrow[r, "d_{\bot}"'] &  \mathcal{M}_{n+i}
    \end{tikzcd}
\end{center}
are pullbacks because $\mathcal{M}$ is a decomposition space, $d_{\bot}$ is an inert map and $d_{n+1}$ and $d_{n}$ are always inner coface maps thus active maps. For the remaining squares, we use that the squares
\begin{center}
    \begin{tikzcd}
        \mathcal{M}_{i+1+n+1} \arrow[r, "d_{i+1}"] \arrow[d, "d_{\top}"'] 
            & \mathcal{M}_{i+1+n} \arrow[d, "d_{\top}"]\\
         \mathcal{M}_{i+1+n} \arrow[r, "d_{i+1}"'] &  \mathcal{M}_{i+1+n-1}
    \end{tikzcd}
\end{center}
are also pullbacks because $\mathcal{M}$ is a decomposition space.

(2) 
To establish the stability condition, by Lemma~\ref{stablelemma} it is enough to prove that the two following squares are pullbacks:
\begin{center}
    \begin{tikzcd}
        B_{0,0} & B_{0,1} \arrow[l, "d_0"'] \\
         B_{1,0} \arrow[u, "e_0"]  & B_{1,1} \arrow[l, "d_0"] \arrow[u, "e_0"'] \arrow[ul, phantom, "\ulcorner", very near start],
    \end{tikzcd}
 \qquad
    \begin{tikzcd}
        B_{0,0} & B_{0,1} \arrow[l, "d_1"'] \\
         B_{1,0} \arrow[u, "e_1"]  & B_{1,1} \arrow[l, "d_1"] \arrow[u, "e_1"'] \arrow[ul, phantom, "\ulcorner", very near start].
    \end{tikzcd}
\end{center}
We can prove this with the same strategy used above. The decomposition space axioms used here are that the following squares are pullbacks
\begin{center}
    \begin{tikzcd}
        \mathcal{M}_{0} \ar[r, "s_0"] 
            & \mathcal{M}_{1} \\
        \mathcal{M}_{1} \ar[r, "s_1"'] \ar[u, "d_\bot"] \ar[ur, phantom, "\urcorner", very near start] & \mathcal{M}_{2} \ar[u, "d_\bot"'],
    \end{tikzcd} \qquad
    \begin{tikzcd}
    \mathcal{M}_{0} \ar[r, "s_0"] 
        & \mathcal{M}_{1} \\
    \mathcal{M}_{1} \ar[r, "s_0"'] \ar[u, "d_\top"] \ar[ur, phantom, "\urcorner", very near start] & \mathcal{M}_{2} \ar[u, "d_\top"'].
\end{tikzcd}
\end{center}

(3) and (4)
The augmentations $Y_j := B_{-1,j}$ and $X_i := B_{i,-1}$ are also given by the following fibre sequences 
\begin{center}
\begin{tikzcd}
     Y_j \arrow[r, ""] 
       \arrow[d, ""'] \arrow[dr, phantom, "\lrcorner", very near start] & \Map(\Delta^{j}, \mathcal{M}) \arrow[d, "S"]\\
    1 \arrow[r, "\name{d_0}"'] & \Map(\Delta^0, \Delta^1),
\end{tikzcd} \qquad
\begin{tikzcd}
     X_i \arrow[r, ""] 
       \arrow[d, ""'] \arrow[dr, phantom, "\lrcorner", very near start] & \Map(\Delta^{i}, \mathcal{M}) \arrow[d, "T"]\\
    1 \arrow[r, "\name{d_1}"'] & \Map(\Delta^0, \Delta^1),
\end{tikzcd}
\end{center}
where the map $S$ sends $\sigma$ to $p \circ \sigma \circ (d^\top)^{j}$, 
and the map $T$ sends $\sigma$ to $p \circ \sigma \circ (d^\bot)^{i}$.
Since the following squares commute
\begin{center}
\begin{tikzcd}
        \Map(\Delta^{i+1+j}, \mathcal{M}) \arrow[r, "(d_{\bot})^{i+1}"] \arrow[d, "R_{i+1+j}"'] & \Map(\Delta^{j}, \mathcal{M}) \arrow[d, "S"] \\
        \Map(\Delta^{1}, \Delta^{1})  \arrow[r, "d_\bot"'] & \Map(\Delta^{0}, \Delta^{1}),
\end{tikzcd} \quad
\begin{tikzcd}
        \Map(\Delta^{i+1+j}, \mathcal{M}) \arrow[r, "(d_\top)^{j+1}"] \arrow[d, "R_{i+1+j}"'] & \Map(\Delta^{i}, \mathcal{M}) \arrow[d, "T"] \\
        \Map(\Delta^{1}, \Delta^{1})  \arrow[r, "d_\top"] & \Map(\Delta^{0}, \Delta^{1}),
\end{tikzcd}
\end{center}
it is enough to define maps $B_{i,j} \to Y_j$ and $B_{i,j} \to X_i$.

The augmentation maps are culf: we need to prove that the back squares of the following cubes are pullbacks:

\begin{center}
\begin{tikzcd}[column sep=tiny]
     Y_0 \arrow[rr, "s_0"] \arrow[dr, ""'] 
        & & Y_{1}  \arrow[dr, ""'] &\\
       & \Map(\Delta^{0}, \mathcal{M})  && \Map(\Delta^{1}, \mathcal{M})  \\
       B_{0,0} \arrow[rr, "s_0" near start] \arrow[dr, ""'] \ar[uu, "d_0" near start] && B_{0,1} \arrow[dr, ""'] \ar[uu, "d_0"' near start] \arrow[from=ul, to=ur, crossing over, "s_0" near start]& \\
       & \Map(\Delta^{1}, \mathcal{M}) \ar[uu, crossing over, "d_0" near start] \arrow[rr, "s_1"] && \Map(\Delta^{2}, \mathcal{M}), \ar[uu, "d_0"' near start]
\end{tikzcd}
\end{center}

\begin{center}
\begin{tikzcd}[column sep=tiny]
     Y_1  \arrow[dr, ""'] 
        & & Y_2  \arrow[dr, ""'] \ar[ll, "d_1"'] &\\
       & \Map(\Delta^{1}, \mathcal{M})  && \Map(\Delta^{2}, \mathcal{M}) \\
       B_{0,1}  \arrow[dr, ""'] \ar[uu, "d_0" near start] && B_{0,2} \ar[ll, "d_1"' near start] \arrow[dr, ""'] \ar[uu, "d_0"' near start] \ar[from=ur, to=ul, crossing over, "d_1"' near start] & \\
       & \Map(\Delta^{2}, \mathcal{M}) \ar[uu, crossing over, "d_0" near start]  && \Map(\Delta^{3}, \mathcal{M}) \ar[ll, "d_2"] \ar[uu, "d_0"' near start].
\end{tikzcd}
\end{center}
We can apply the Lemma~\ref{pbksqfibresq} since the front square is a pullback because $\mathcal{M}$ is a decomposition space.
\end{proof}

To summarise, given a decomposition space correspondence $p: \mathcal{M} \to \Delta^1$, we get a bicomodule configuration and then $\infGrpd_{/B_{0,0}}$ is a bicomodule by Theorem~\ref{thm:bicomodule}.

\subsection{Cocartesian and cartesian fibrations of decomposition spaces}
\label{sub:cartsset}

Ayala and Francis \cite{AF} formulate a homotopy-invariant definition of cartesian and cocartesian fibrations so it can be equally
well formulated in any model for $\infty$-categories. We adapt here those definitions to decomposition spaces.


Let $p: X \to Y$ be a simplicial map between decomposition spaces.
A morphism $\Delta^1 \xra{<s \xra{a} t>} X$ is \emph{$p$-cocartesian} if the diagram of coslices of decomposition spaces
\begin{center}
    \begin{tikzcd}
        ^{a/}X \arrow[r, ""] 
           \arrow[d, ""'] & ^{s/}X \arrow[d, ""]\\
        ^{pa/}Y \arrow[r, ""'] & ^{ps/}Y
    \end{tikzcd}
\end{center}
is a pullback, where the coslice $^{s/}X$ is given by pullback of lower décalage $\Dec_{\bot}(X)$:
\begin{center}
    \begin{tikzcd}
        (^{s/}X)_n \arrow[r, ""] 
           \arrow[d, ""'] \ar[dr, phantom, "\lrcorner", very near start] & \Dec_{\bot}(X)_{n} \arrow[d, "(d_\top)^{n+1}"]\\
        1 \arrow[r, "\name{s}"'] & X_0,
    \end{tikzcd}
\end{center}
similarly the coslice $^{a/}X$ is given by pullback of $\Dec_{\bot}(\Dec_{\bot}(X))$
\begin{center}
    \begin{tikzcd}
        (^{a/}X)_n \arrow[r, ""] 
           \arrow[d, ""'] \ar[dr, phantom, "\lrcorner", very near start] & \Dec_{\bot}(\Dec_{\bot}(X))_{n} \arrow[d, "(d_\top)^{n+1}"]\\
        1 \arrow[r, "\name{a}"'] & X_1,
    \end{tikzcd}
\end{center}
and the functor ${}^{a/}X \to {}^{s/}X$ is given by $\Dec_\bot(d_\bot)$, where the simplicial map $d_\bot: \Dec_\bot(X) \to X$ is given by the original $d_0$.

The functor $p: X \to Y$ is a \emph{cocartesian fibration} if any diagram of solid arrows
\begin{center}
    \begin{tikzcd}
        \Delta^0 \arrow[r, ""] 
           \arrow[d, "d_1"'] & X \arrow[d, "p"]\\
        \Delta^1 \arrow[r, ""'] \ar[ur, dotted] & Y
    \end{tikzcd}
\end{center}
admits a $p$-cocartesian diagonal filler.

Similarly, a morphism $\Delta^1 \xra{<s \xra{a} t>} X$ is \emph{$p$-cartesian} if the diagram of slice decomposition spaces
\begin{center}
    \begin{tikzcd}
        X_{/a} \arrow[r, ""] 
           \arrow[d, ""'] & X_{/t} \arrow[d, ""]\\
        Y_{/pa} \arrow[r, ""'] & Y_{/pt}
    \end{tikzcd}
\end{center}
is a pullback, where the slice $X_{/t}$ is given by pullback of the upper décalage $\Dec_{\top}(X)$, the slice $X_{/a}$ is given by pullback of $\Dec_{\top}(\Dec_{\top}(X))$, and the functor $X_{/a} \to X_{/t}$ is given by $\Dec_\top(d_\top)= d_{\top-1}$, where $d_\top: \Dec_\top(X) \to X$ is given by the original $d_\top$.

The functor $p: X \to Y$ is a \emph{cartesian fibration} if any diagram of solid arrows
\begin{center}
    \begin{tikzcd}
        \Delta^0 \arrow[r, ""] 
           \arrow[d, "d_0"'] & X \arrow[d, "p"]\\
        \Delta^1 \arrow[r, ""'] \ar[ur, dotted] & Y
    \end{tikzcd}
\end{center}
admits a $p$-cartesian diagonal filler.

\subsubsection*{Bisimplex category with diagonal maps}

We define $\overline{[i,j]} : {M}_{\phi_{i,j}} \to \Delta^{1}$ to be the canonical projection from the mapping cylinder of $\phi_{i,j}:= (d^\top)^j : \Delta^{i} \to \Delta^{i+j}$; it is a
cocartesian fibration. They assemble into a category, denoted $\overline{\simplexcat_{/\Delta^1}}$, of shape like $\simplexcat_{/ \Delta^1}$, but with extra diagonal maps $d: \overline{[i{-}1,j]} \to \overline{[i,j{-}1]}$ given by inclusion. These satisfy new simplicial identities:
\revised{
$\sigma_k d = d \sigma_{k+1}$, $0\le k \le j$, where $\sigma_k$ are degeneracy maps “on $j$” (horizontal) and $d$ are diagonal maps, and with face maps:  $d \delta_{k+1} = \delta_k d$, $0\le k\le j$, where $\delta_k$ are horizontal face maps.
Similarly for degeneracy maps $\tau_k$ “on $i$” (vertical), $\tau_k d = d \tau_{k}$, $0\le k < i$ and 
$d \epsilon_{k} = \epsilon_k d$, $0\le k < i$, where $\epsilon_k$ are vertical face maps.
} 
For example, we can draw $\overline{[2,1]}$ as follows
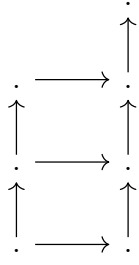
\begin{center}
  \begin{tikzcd}\label{picturefilled}
                    & .        \\
    . \ar[r]        & . \ar{u} \\
    . \ar[r] \ar[u] & . \ar[u] \\
    . \ar[r] \ar[u] & . \ar[u] 
  \end{tikzcd}
\end{center}
where the horizontal maps lie over the map in $\Delta^1$.
It is like a cocartesian version of the earlier drawing of Remark~\ref{shape1}.

\begin{rmk}
    We can draw $\overline{[i,j]}$ as a column of $i{+}1$ black dots followed by $j{+}1$ white dots. Where arrows in $\simplexcat_{/\Delta^1}$ send black dots to black dots and white dots to white dots (without crossing), in $\overline{\simplexcat_{/\Delta^1}}$ we allow moreover to map white dots to black dots.
\end{rmk}

There is a natural notion of nerve in the context of cocartesian fibrations over $\Delta^1$: given a cocartesian fibration $p : \mathcal{M} \to \Delta^1$ between decomposition spaces, define the cocartesian nerve $N_{\text{cocart}}: \Cat_{\infty/\Delta^1}^{\text{cocart}} \to \Fun((\overline{\simplexcat_{/\Delta^1}})\op, \infGrpd)$ by
$N_{\text{cocart}}(p)_{\overline{i,j}} := \Map_{/\Delta^1}^{\text{cocart}}(\overline{[i,j]}, \mathcal M)$, the mapping space preserving cocartesian arrows.

Similarly to the previous Section~\ref{sub:corresp}, we get a bicomodule configuration and $\infGrpd_{/B_{0,0}}$ is a bicomodule over $\infGrpd_{/X_{1}}$ and $\infGrpd_{/Y_{1}}$. We have here moreover diagonal maps $B_{i,j-1} \to B_{i-1,j}$ and new sections $s_{-1}: B_{i,j-1} \to B_{i,j}$, for $i \ge 0 $ given by the composition with a diagonal map. That is $\infGrpd_{/B_{0,0}}$ is pointed as a right comodule over $\infGrpd_{/Y_{1}}$.

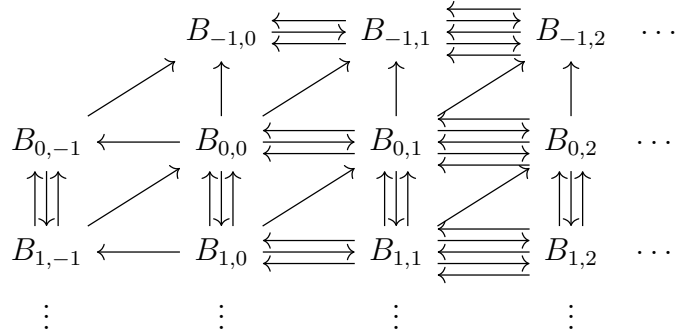
\begin{center}
    \begin{tikzcd} 
         & B_{-1,0} \ar[r] 
         & B_{-1,1} \ar[l, shift left=1.5]
                    \ar[l, shift right=1.5]
                    \ar[r, shift left=1.5] 
                    \ar[r, shift right=1.5]
         & B_{-1,2} \ar[l]
                    \ar[l, shift left=3] 
                    \ar[l, shift right=3] 
                    \ar[r, phantom, "\dots"]
                    & \phantom{} \\
       B_{0,-1} \ar[d]
                \ar[ur]
         & B_{0,0} \ar[d]
                   \ar[u]
                   \ar[l] 
                   \ar[r] 
                   \ar[ur]
         & B_{0,1} \ar[d]
                   \ar[u]
                   \ar[l, shift left=1.5]
                   \ar[l, shift right=1.5]
                   \ar[r, shift left=1.5] 
                   \ar[r, shift right=1.5]
                   \ar[ur] 
         & B_{0,2} \ar[u]
                   \ar[d]
                   \ar[l]
                   \ar[l, shift left=3] 
                   \ar[l, shift right=3]
                   \ar[r, phantom, "\dots"]
                    & \phantom{}\\
       B_{1,-1} \ar[u, shift left=1.5]
                \ar[u, shift right=1.5]
                   \ar[ur]
                   \ar[d, phantom, "\vdots"]
         & B_{1,0} \ar[u, shift left=1.5]
                   \ar[u, shift right=1.5]
                   \ar[l] 
                   \ar[r] 
                   \ar[ur]
                   \ar[d, phantom, "\vdots"]
         & B_{1,1} \ar[u, shift left=1.5]
                   \ar[u, shift right=1.5]
                   \ar[l, shift left=1.5]
                   \ar[l, shift right=1.5]
                   \ar[r, shift left=1.5] 
                   \ar[r, shift right=1.5] 
                   \ar[ur]
                   \ar[d, phantom, "\vdots"]
         & B_{1,2}   
                   \ar[u, shift left=1.5]
                   \ar[u, shift right=1.5]
                   \ar[l]
                   \ar[l, shift left=3] 
                   \ar[l, shift right=3]
                   \ar[r, phantom, "\dots"]
                   \ar[d, phantom, "\vdots"]
                    & \phantom{}\\
        \phantom{}
          & \phantom{}\
          &\phantom{}
          & \phantom{}
    \end{tikzcd}
\end{center}






\revised{
We now adapt Lurie's definition of adjunction of $\infty$-categories \cite{Lurie} to decomposition spaces.
}

An \emph{adjunction} between decomposition spaces $X$ and $Y$ is a simplicial map between decomposition spaces $p: \mathcal{M} \to \Delta^1$ which is both a cartesian and a cocartesian fibration together with equivalences $X \eq \mathcal{M}_{\{0\}}$ and $Y \eq \mathcal{M}_{\{1\}}$.

\revised{
\begin{prop}
    Given an adjunction $p: \mathcal{M} \to \Delta^1$, the bisimplicial $\infty$-groupoid $B$ described above is a bicomodule configuration. Moreover $\infGrpd_{/B_{0,0}}$ is pointed as a right comodule over $\infGrpd_{/Y_{1}}$, and as a left comodule over $\infGrpd_{/X_{1}}$.
\end{prop}
}
\revised{
\begin{proof}
    By Proposition~\ref{runex}, the bisimplicial $\infty$-groupoid $B$ is a bicomodule configuration. The pointings are given by the cartesian and cocartesian conditions.
\end{proof}

    Adjunctions between $\infty$-categories are examples of adjunctions between decomposition spaces.
}

\begin{ex}\label{exFG}
\revised{
  To illustrate these abstract concepts, let us spell out the bicomodule configuration associated to an ordinary adjunction of $1$-categories $F\colon X \rightleftarrows Y\colon G$.
The $\infty$-groupoid $B_{0,0}$ is now just the set of arrows $Fx \to y$, which by adjunction correspond to arrows $x \to Gy$, and $B_{1,0}$ is the set of composable pair $Fx \to Fx' \to y$ (which is the same as $x \to x' \to Gy$). In general $B_{i,j}$ is the set of chains of composable arrows $Fx_0 \to \dots \to Fx_i \to y_j \to \dots \to y_0$.
These chains can be drawn as in the picture page \pageref{picturejump}, where the horizontal arrow is a `mixed' arrow $Fx \to y$.
This drawing can be filled as in the picture page~\pageref{picturefilled}, and thus giving a right pointing by the following rearrangement:
}
\begin{center}
  \begin{tikzcd}
    Fx' \ar[r] & y \\
    Fx  \ar[u] &  
  \end{tikzcd}
$\quad \mapsto \quad$
  \begin{tikzcd}
              & y         \\
    Fx \ar[r] & Fx' \ar{u}
  \end{tikzcd}
\end{center}
\revised{
Similarly the functor $G: Y \to X$ induces a left pointing on the equivalent (by adjunction) augmented simplicial set of chains $x_0 \to \dots \to x_i \to Gy_j \to \dots \to Gy_0$.
}
\end{ex}



\section{Möbius inversion for comodules and a Rota formula}
\label{sec:mobcomod}

\subsection{Finiteness and cardinality}
\label{sub:moreprelim}


An $\infty$-groupoid $X$ is \emph{locally finite} if at each base point $x$ the homotopy groups $\pi_i(X, x)$ are finite for $i \ge 1$ and are trivial for $i$ sufficiently large. It is called \emph{finite} if furthermore it has only finitely many components. We denote $\infgrpd$ (following the notation of \cite{GKT2}) the $\infty$-category of finite $\infty$-groupoids. 
A map is \emph{finite} if each fibre is finite. 
A pullback of any homotopy finite map is again finite.
A span $I \xleftarrow{p} M \xrightarrow{q} J$ and the corresponding linear functor $\infGrpd_{/I} \xra{} \infGrpd_{/J}$
are \emph{finite} if the map $p$ is finite.

A decomposition space $X$ is \emph{locally finite} if $X_1$ is locally finite and both $s_0$ and $d_1$ are finite maps \cite[\S 7.4]{GKT2}.

\begin{prop}[{\cite[Proposition 4.3]{GKT:HLA}}]\label{prop:finitelinfunc}
    Let $I$, $J$, $M$ be locally finite $\infty$-groupoids and
    $I \xleftarrow{p} M \xrightarrow{q} J$
    a finite span. Then the induced finite linear functor
    $\infGrpd_{/I} \xra{} \infGrpd_{/J}$
    restricts to
    $\infgrpd_{/I} \xra{} \infgrpd_{/J}$.
\end{prop}



The cardinality \cite{BD} of a finite $\infty$-groupoid $X$ is the alternating product of cardinalities of the homotopy groups
\[
    |X| = \sum_{x \in \pi_0 (X)} \prod_{k=1}^{\infty} |\pi_k(X,x)|^{(-1)^{k}}.
\]
For a locally finite $\infty$-groupoid $S$, there is a notion of cardinality $|-|: \infgrpd_{/S} \to \Q_{\pi_0 S}$ sending a basis element $\name{s}$ to the basis element $\delta_s = |\name{s}|$.




For any locally finite decomposition space $X$, we can take the cardinality of the linear functors
$
   \delta: \infgrpd_{/X_1} \xra{} \infgrpd
$
and
$
    \Delta: \infgrpd_{/X_1} \xra{} \infgrpd_{/X_1 \times X_1}
$
to obtain a coalgebra structure
\begin{align*}
    \Q_{\pi_0 X_1} & \xra{|\delta|} \Q \\
    \Q_{\pi_0 X_1} & \xra{|\Delta|} \Q_{\pi_0 X_1} \otimes \Q_{\pi_0 X_1}
\end{align*}
called the \emph{numerical incidence coalgebra} of $X$, see
\cite[\S 7.7]{GKT2}.


\subsection{Completeness and Möbius condition}
\label{sub:complete}

A decomposition space is called \emph{complete} if $s_0: X_0 \to X_1$ is a monomorphism \cite[\S 2]{GKT2}.
Since $s_0$ is a monomorphism, we can identify $X_0$ with a $\infty$-subgroupoid of $X_1$. We denote $X_a$ its complement:
$X_1 \eq X_0 + X_a$.
More generally, recall the word notation from \cite{GKT2}: consider the alphabet with three letters $\{ 0,1,a\}$; $0$ indicates degenerate edges $s_0(x) \in X_1$, $a$ denotes edges specified to be nondegenerate, and $1$ denotes unspecified edges.
For $w$ a word of length $n$ in this alphabet, define
\[
  X^{w} = \prod_{i \in w} X_i \subset (X_1)^n.
\]
This inclusion is full since $X_a \subset X_1$ is full by completeness.

Denote by $X_w$ the $\infty$-groupoid of $n$-simplices whose principal edges have the types indicated in the word $w$, that is the full subgroupoid of $X_n$ given by the following pullback

\begin{center}
    \begin{tikzcd}
        X_w \arrow[r, ""] 
            \arrow[d, ""'] 
            \ar[dr, phantom, "\lrcorner", very near start]
             & X_n \arrow[d, ""]\\
        X^{w} \arrow[r, ""'] & (X_1)^{n}.
    \end{tikzcd}
\end{center}

We define $\ora{X}_{n} = X_{a\dots a} \subset X_{n}$ to be the full subgroupoid of simplices with all principal edges nondegenerate.
\revised{It is the complement of the union of the essential images of the degeneracy maps 
$s_i\colon X_{n-1} \to X_n$, that is
\[
\ora X_n = X_n \setminus \bigcup_{i=0}^{n-1} \Im(s_i) .
\]
By definition $\ora X_0 = X_0$.
}


For a complete decomposition space, the spans
$
    X_1 \xla{d_1^{n-1}} \ora{X}_n \rightarrow 1
$
define linear functors, the \emph{Phi functors}
\[
    \Phi_n: \infGrpd_{/X_1} \to \infGrpd.
\]
We also put  $\displaystyle \Phi_{\text{even}} := \sum_{n \text{ even}} \Phi_{n}$, and $ \displaystyle \Phi_{\text{odd}} := \sum_{n \text{ odd}} \Phi_{n}$.

The incidence algebra of a decomposition space contains the \emph{zeta functor} 
\[
    \zeta: \infGrpd_{/X_1} \to \infGrpd
\]
given by the span
$X_1 \xleftarrow{=} X_1 \xrightarrow{} 1$.



\begin{thm}[{\cite[Theorem 3.8]{GKT2}}]\label{mobdecomp}
    For a complete decomposition space, the following Möbius inversion holds:
\begin{align*}
        \zeta \ast \Phi_{\text{even}} & \eq \delta + \zeta \ast \Phi_{\text{odd}} \\
      \eq \Phi_{\text{even}} \ast \zeta & \eq \delta + \Phi_{\text{odd}} \ast \zeta.
\end{align*}
\end{thm}

This is however not enough to allow the Möbius inversion formula to descend to the vector space level.
A complete decomposition space $X$ is of \emph{locally finite length} \cite{GKT2} if every edge $f \in X_1$ has finite length, that is, the fibres $(\ora{X}_n)_f$ of $d_1^{(n)}: \ora{X}_n \to X_1$ over $f$ are empty for $n$ sufficiently large.


A \emph{Möbius decomposition space} \cite{GKT2} is a decomposition space which is locally finite and of locally finite length; the fibre $(\ora{X}_n)_f$ is finite (eventually empty).
It follows that the map
\[
    \sum_n d_1^{n-1}: \sum_n \ora{X}_n \to X_1
\]
is finite;
by Proposition~\ref{prop:finitelinfunc}, the Phi functors descend to 
\[\Phi_n: \infgrpd_{/X_1} \to \infgrpd\]
and we can take cardinality to obtain functions $|\Phi_n|: \Q_{\pi_0 X_1} \to \Q$.


Finally, we can take cardinality of the abstract Möbius inversion formula of \ref{mobdecomp}, see \cite{GKT2} for a complete exposition.

\begin{thm}[{\cite[Theorem 8.9]{GKT2}}]
    If $X$ is a Möbius decomposition space, then the cardinality of the zeta functor, $|\zeta|: \Q_{\pi_0 X_1} \to \Q$, is convolution invertible with inverse $|\mu| := |\Phi_{\text{even}}| - |\Phi_{\text{odd}}|$:
    \[
        |\zeta| \ast |\mu| = |\delta| = |\mu| \ast |\zeta|.
    \]
\end{thm}


\subsection{Right and left convolutions}
\label{sub:convolution}

We introduce left and right convolution actions as dual to the comodule structures. Explicitly, given a right comodule configuration $C \to Y$, we get a right comodule $\infGrpd_{/C_{0}}$ over $\infGrpd_{/Y_{1}}$.
The \emph{right convolution} $\theta \star_r \beta$ of the two functors
$\theta: \infGrpd_{/C_{0}} \to \infGrpd$ and
$\beta: \infGrpd_{/Y_{1}} \to \infGrpd$,
given by the spans
 $C_{0} \xla{} M \xra{} 1$ and $Y_{1} \xla{} N \xra{} 1$, 
 is the composite of $\theta \otimes \beta$ with the right coaction $\gamma_r$:
\[\theta \star_r \beta: \infGrpd_{/C_{0}} \xra{\gamma_r} \infGrpd_{/C_{0}} \otimes \infGrpd_{/Y_{1}} \xra{\theta \otimes \beta} \infGrpd,\]
 where the tensor product $\theta \otimes \beta$ is given by the span 
$C_{0} \times Y_{1} \xla{} M \times N \xra{} 1$.

Similarly, given a left comodule configuration, we can define the \emph{left convolution} $\alpha \star_l \theta$ of $\alpha : \infGrpd_{/X_{1}} \to \infGrpd$ and $\theta: \infGrpd_{/C_{0}} \to \infGrpd$:
\[\alpha \star_l \theta: \infGrpd_{/C_{0}} \xra{\gamma_l} \infGrpd_{/X_{1}} \otimes \infGrpd_{/C_{0}} \xra{\alpha \otimes \theta} \infGrpd.\]

If we have a bicomodule configuration, then the following associativity formula expresses the compatibility of coactions from Theorem~\ref{thm:bicomodule}.

\begin{cor}\label{associativityconvol}
Given a bicomodule configuration, the convolutions defined above satisfy
\[
    \alpha \star_l (\theta \star_r \beta)
     \eq (\alpha \star_l \theta) \star_r \beta.
\]
\end{cor}


\subsection{Möbius inversion for (co)modules}
\label{sub:phifunctors}

Let $C \to Y$ be a comodule configuration.
The zeta functor 
\[\zeta^{C}: \infGrpd_{/C_{0}} \to \infGrpd\]
is the linear functor defined by the span
\[C_{0} \xleftarrow{=} C_{0} \xrightarrow{} 1.\]

Let $C \to Y$ be a right pointed comodule configuration. 
The augmented simplicial $\infty$-groupoid $C$ is an object of the functor $\infty$-category
\[
    \Fun{(\simplexcat\op_{\text{bot}}, \infGrpd)}
\]
where $\simplexcat_{\text{bot}}$ is the simplex category of finite linear orders with a specified bottom element, and with monotone maps preserving the bottom element.
The forgetful functor $\simplexcat_{\text{bot}} \to \simplexcat$ is right adjoint to the functor $j: \simplexcat \to \simplexcat_{\text{bot}}$ adding a bottom element.

\revised{
\begin{rmk}
    In the situation where $Y$ is Segal and $C = \Dec_\bot Y$, we can take $C_{-1}$ to be $Y_0$, with $d_0$ as augmentation map. By \cite[Lemma 6.1.3.16]{Lurie}, this is a colimit diagram.
\end{rmk}
}

A right pointed comodule configuration $f: C \to Y$ is \emph{complete} if the new degeneracies $s_{-1}: C_{n-1} \to C_{n}$ are monomorphisms.
\revised{
Since $s_{-1}$ is a monomorphism, we can identify $C_{-1}$ with a $\infty$-subgroupoid of $C_0$. We denote by $C_b$ its complement: $C_{0} = C_{-1} + C_b$. 
Denote by $C_{vw}$ the $\infty$-groupoid of simplices whose principal edges have the type indicated in the word $vw
$, where $v \in \{-1,0,b\}$ and $w$ is a word in the alphabet $\{0, 1,a\}$, that is the full $\infty$-subgroupoid of $C_n$ given by the pullback
}
\begin{center}
    \begin{tikzcd}
        C_{vw} \arrow[r, ""] 
               \arrow[d, ""'] 
               \ar[dr, phantom, "\lrcorner", very near start]
             & C_n \arrow[d, ""]\\
        C_{v} \times Y^{w} \arrow[r, ""'] & C_{0} \times (Y_1)^{n},
    \end{tikzcd}
\end{center}
\revised{where $n =|w| \geq 0$.}
The principal edges of the $\infty$-groupoid $C_n$ consist of an element in $C_0$ given by $(d_\top)^n$, and $n$ edges in $Y_1$, the principal edges of the image of $C_n$ by $f$.
In this situation, we define \revised{$\ora{C}_{n} = C_{ba\dots a} \subset C_{n}$} to be the full subgroupoid of simplices with all principal edges nondegenerate.
\revised{It is given by the pullback diagram}
\begin{center}
    \begin{tikzcd}
        C_{ba\dots a} \ar[r, ""] 
           \ar[d, ""'] 
           \ar[dr, phantom, "\lrcorner", very near start] & C_{n} \ar[d, ""]\\
        C_{b} \times Y^{a \dots a} \ar[r, ""'] & C_{0} \times Y_1^n.
    \end{tikzcd}
\end{center}


Define
\[\delta^{R}: \infGrpd_{/C_{0}} \to \infGrpd\]
to be the linear functor given by the span
\[C_{0} \xleftarrow{s_{-1}} C_{-1} \xrightarrow{} 1\]
and define the \emph{right Phi functors}
\[\Phi_{n}^{R}: \infGrpd_{/C_{0}} \to \infGrpd\]
to be the linear functors given by the spans
\[C_{0} \xleftarrow{} \ora{C}_{n} \xrightarrow{} 1.\]
If $n = -1$, 
$\ora{C}_{-1} = C_{-1}$ (by convention)
and $\Phi_{-1}^{R}$ is the linear functor $\delta^{R}$.


\begin{lem}\label{pbkC}
\revised{For every word $w$ in the alphabet $\{0, 1,a\}$}, the following square is a pullback:
    \begin{center}
        \begin{tikzcd}
            C_{0w} \ar[r, ""] 
               \ar[d, ""'] 
               \ar[dr, phantom, "\lrcorner", very near start] & C_{1} \ar[d, "(d_{\top}{,}f)"]\\
           C_{0} \times Y_w \ar[r, ""'] & C_{0} \times Y_1.
        \end{tikzcd}
    \end{center}
\end{lem}

\begin{proof}
Let $n = |w|$.
The square is the top rectangle of the diagram

\begin{center}
  \begin{tikzcd}
      C_{0w} \ar[r, ""] \ar[d, ""'] 
        & C_{n} \ar[r, ""] \ar[d, ""] \ar[dr, phantom, "\lrcorner", very near start]
        & C_{1} \ar[d, "(d_{\top}{,}f)"]\\
      C_{0} \times Y_w \ar[r, ""'] \ar[d, ""] \ar[dr, phantom, "\lrcorner", very near start]
        & C_{0} \times Y_n \ar[r, ""] \ar[d, ""]
        & C_{0} \times Y_1\\
          C_{0} \times Y^{w} \ar[r, ""]
        & C_{0} \times Y_1^{n}.
        &
  \end{tikzcd}
\end{center}
The bottom square and left-hand rectangle are pullbacks by definition of $Y_{w}$ and $C_{0w}$, hence the top left-hand square is a pullback.
The right-hand square is a pullback because the augmentation map $C \to Y$ is culf.
Hence the top rectangle, which is the desired square, is a pullback.
\end{proof}



Given a complete decomposition space $Y$, we denote $\Phi_{n}^Y : \infGrpd_{/Y_{1}} \to \infGrpd$ the usual Phi functors, see \ref{sub:moreprelim} above.

\begin{prop}\label{prop:Mobformula}
    The right Phi functors satisfy 
    \[\zeta^{C}\star_r \Phi_{n}^Y \eq \Phi_{n-1}^{R} + \Phi_{n}^{R}.\]
\end{prop}

\begin{proof}
Compute the convolution action $\zeta^{C}\star_r \Phi_{n}^Y$ by Lemma~\ref{pbkC} as:

\begin{center}
    \begin{tikzcd}[column sep=large]
    C_{0} & & \\
        C_{1}  \arrow[u, ""]
           \arrow[d, ""'] 
             & C_{0a\dots a} \ar[dl, phantom, "\llcorner", very near start] 
                         \arrow[l, ""] 
                         \arrow[d, ""] 
                         \arrow[ul, ""] 
                         \arrow[dr, ""]& \\
        C_{0} \times Y_1  & C_{0} \times \ora{Y_n} \arrow[l, ""'] \arrow[r, ""'] & 1.
    \end{tikzcd}
\end{center}
But $C_{0 a \dots a} \eq C_{-1 a \dots a} + C_{b a \dots a} \eq \ora{C}_{n-1} + \ora{C}_{n}$.
This is an equivalence of $\infty$-groupoids over $C_{0}$ and the resulting span is $\Phi_{n-1}^{R} + \Phi_{n}^{R}$.
\end{proof}

Denote 
\[
    \Phi_{\text{even}}^Y := \sum_{n \text{ even}} \Phi_{n}^Y, \quad \Phi_{\text{odd}}^Y := \sum_{n \text{ odd}} \Phi_{n}^Y.
\]
The previous proposition implies the following Möbius inversion formula.

\begin{thm}\label{thm:Mobinversion}
    Given $C \to Y$ a complete right pointed comodule configuration, 
    \[\zeta^{C} \star_r \Phi_{\text{even}}^Y \eq \delta^{R} + \zeta^{C} \star_r \Phi_{\text{odd}}^Y.\]
\end{thm}

\begin{proof}
The two linear functors are equivalent to the sum of the right Phi functors:
    \[\zeta^{C} \star_r \Phi_{\text{even}}^Y 
    \eq \Phi_{-1}^{R} + \Phi_{0}^{R} + \Phi_{1}^{R} + \cdots
    \eq \delta^{R} + \zeta^{C} \star_r \Phi_{\text{odd}}^Y.\]
\end{proof}

We can also define a \emph{left pointed} comodule configuration $D \to X$, with new top sections instead of bottom: we consider instead the mapping cylinder of $\simplexcat \to \simplexcat_{\text{top}}$, where $\simplexcat_{\text{top}}$ is the simplex category of finite linear orders with a specified top element, and with monotone maps preserving the top element.
A left pointed comodule configuration is \emph{complete} if the new degeneracies $t_{\top+1} : D_{n-1} \to D_n$ are monomorphisms. Similarly, we define the \emph{left Phi functors} and $\delta^{L}$ using $t_{\top+1}$ and $e_{\top}$ and we obtain the following formula.

\begin{thm}\label{thm:Mobinversionleft}
    Given $D \to X$ a complete left pointed comodule configuration,
    \[\Phi_{\text{even}}^X \star_l \zeta^{D} \eq \delta^{L} +  \Phi_{\text{odd}}^X \star_l \zeta^{D}.\]
\end{thm}

\subsection{Möbius bicomodule configurations and the Rota formula}
\label{sub:Mobbicomod}

In order to take homotopy cardinality  to recover the usual Möbius inversions, we need to impose some finiteness conditions. We adapt the approach of \cite{GKT2} summarised in \ref{sub:moreprelim} and \ref{sub:complete} above.

A \emph{right Möbius comodule configuration} is a complete right pointed comodule configuration $C \to Y$ such that the decomposition space $Y$ is Möbius and the augmented comodule is Möbius, that is
\begin{itemize}
    \item $C$ is locally finite: the $\infty$-groupoid $C_0$ is locally finite and both $s_{-1}$ and $d_0$ are finite maps;
    \item $C$ is of locally finite length: every edge has a finite length, that is for all $a \in C_0$, the fibres of $d_0^{(n)}: \ora{C}_{n} \to C_0$ over $a$ are empty for $n$ sufficiently large.
\end{itemize}

Under these conditions, the Phi functors descend to 
\[
\Phi_{n}^{R} : \infgrpd_{/C_{0}} \to \infgrpd
\]
and we can now take the cardinality of the “Möbius formulas” (Theorems~\ref{thm:Mobinversion} and \ref{thm:Mobinversionleft}).

Similarly we define a \emph{left Möbius comodule configuration} to be a complete left pointed comodule configuration $D \to X$ such that the decomposition space $X$ is Möbius and the augmented comodule is Möbius, using $t_{\top+1}$ and $d_{\top}$.

\begin{thm}\label{cardmobform}
Given $C \to Y$ a  right Möbius comodule configuration and $D \to X$ a left Möbius comodule configuration,
    \[
    |\zeta^{C}| \star_r |\mu^Y|  = |\delta^{R}|, \qquad  |\mu^X| \star_l |\zeta^{D}|  = |\delta^{L}|,\]
where $|\mu^Y| := |\Phi^Y_{\text{even}}| - |\Phi^Y_{\text{odd}}|$
and $|\mu^X| := |\Phi^X_{\text{even}}| - |\Phi^X_{\text{odd}}|$.
\end{thm}

A \emph{Möbius bicomodule configuration} is a bicomodule configuration with two pointings such that both left and right comodule configurations are Möbius.
It hence has extra degeneracy maps in both directions, extra bottom degeneracy maps in horizontal direction and extra top degeneracy maps in vertical direction.

Note that given a Möbius bicomodule configuration $B$, the zeta functors are defined only on the $\infty$-groupoid $B_{0,0}$ and then are the same for the two comodules.
In both cases it is given by the span
\[B_{0,0} \xleftarrow{=} B_{0,0} \xrightarrow{} 1.\]

\begin{thm}\label{rotaformulabicomod}
Given a Möbius bicomodule configuration $B$ with $X := B_{\bullet, -1}$ and $Y:=B_{-1, \bullet}$, we have
\[
    |\mu^{X}| \star_l |\delta^R| = |\delta^{L}| \star_r |\mu^Y|,
\]
where $\delta^R$ is the linear functor given by the span
\[
    B_{0,0} \xleftarrow{} X_{0} \xrightarrow{} 1,
\]
and $\delta^{L}$ is the linear functor given by the span
\[
    B_{0,0} \xleftarrow{} Y_{0} \xrightarrow{} 1.
\]
\end{thm}

\begin{proof}
Using the Möbius formulas at the algebraic level from Theorem~\ref{cardmobform}, and the associativity of the convolution actions from Proposition~\ref{associativityconvol}, we compute
    \begin{align*}
      |\mu^X| \star_l |\delta^R| 
        &=  |\mu^X| \star_l (|\zeta| \star_r |\mu^Y|) \\
        &= (|\mu^X| \star_l |\zeta|) \star_r |\mu^Y| \\
        &= |\delta^{L}| \star_r |\mu^Y|.
    \end{align*}
\end{proof}


\subsection{Möbius bicomodule configurations from adjunctions of Möbius decomposition spaces}
\label{sub:runexMob}

We saw in Section~\ref{sub:cartsset} that given a cocartesian fibration $p : \mathcal{M} \to \Delta^1$ between decomposition spaces, we obtain a right comodule configuration $B$, with diagonal maps $B_{i,j-1} \to B_{i-1,j}$ and new sections $s_{-1}: B_{i,j-1} \to B_{i,j}$, for $i \ge 0 $ given by the composition with a diagonal map.

\begin{lem}\label{runexcomplete}
     Given a cocartesian fibration $p: \mathcal{M} \to \Delta^1$ between decomposition spaces, suppose moreover that $\mathcal{M}$ is complete. Then the associated right pointed comodule configuration is complete.
\end{lem}

\begin{proof}
The new sections will be monomorphisms if the following square is a pullback:
\begin{center}
    \begin{tikzcd}
        B_{i,j-1} \arrow[r, "\id"] 
           \arrow[d, "\id"'] & B_{i,j-1} \arrow[d, "s_{-1}"]\\
        B_{i,j-1} \arrow[r, "s_{-1}"'] & B_{i,j}.
    \end{tikzcd}
\end{center}
By assumption, $\mathcal{M}$ is a complete decomposition space, hence all degeneracy maps are monomorphisms, and we can apply Lemma~\ref{pbksqfibresq}, to obtain the desired pullbacks.
\end{proof}


Instantiating the general definitions from Section~\ref{sub:phifunctors},
the zeta functor 
\[\zeta: \infGrpd_{/B_{0,0}} \to \infGrpd\]
is given by the span
\[B_{0,0} \xleftarrow{=} B_{0,0} \xrightarrow{} 1,\]
and the functor
\[\delta^{R}: \infGrpd_{/B_{0,0}} \to \infGrpd\]
is defined by the span
\[B_{0,0} \xleftarrow{s_{-1}} B_{0,-1} \xrightarrow{} 1.\]

The right comodule configuration being complete, we get a Möbius inversion formula (Theorem~\ref{thm:Mobinversion}):
    \[\zeta \star_r \Phi_{\text{even}}^Y \eq \delta^{R} + \zeta \star_r \Phi_{\text{odd}}^Y,\]
where $Y:= B_{-1,\bullet}$.




Similarly, given a cartesian fibration $p : \mathcal{M} \to \Delta^1$ between decomposition spaces, we obtain a left pointed comodule configuration.

\begin{lem}\label{runexleftcomplete}
     Given a cartesian fibration $p: \mathcal{M} \to \Delta^1$ between decomposition spaces, suppose moreover that $\mathcal{M}$ is complete. Then the left pointed comodule configuration is complete.
\end{lem}

The functor
\[\delta^{L}: \infGrpd_{/B_{0,0}} \to \infGrpd\]
is given by the span
\[
    B_{0,0} \xleftarrow{} B_{-1,0} \xrightarrow{} 1.
\]
This leads to the Möbius inversion formula
\[
    \Phi_{\text{even}}^X \star_l \zeta  \eq \delta^{L} +  \Phi_{\text{odd}}^X \star_l \zeta.
\]











Given an adjunction between decomposition spaces, that is a simplicial map $\mathcal{M} \to \Delta^1$ which is both cartesian and cocartesian, and suppose that $\mathcal{M}$ is complete, then we just obtained two Möbius inversion formulas.

\begin{thm}\label{runexRota}
     Given an adjunction of decomposition spaces in the form of a bicartesian  fibration $p: \mathcal{M} \to \Delta^1$, suppose moreover that $\mathcal{M}$ is a Möbius decomposition space. Then the bicomodule configuration extracted from this data is Möbius. In particular, we have the Rota formula for the adjunction $p$:
\[
    |\mu^{X}| \star_l |\delta^R| = |\delta^{L}| \star_r |\mu^Y|.
\]
\end{thm}

\begin{proof}
\revised{
First observe that $B_{0,0}$, and in fact all $B_{i,j}$, are locally finite since they are given by pullback (see page~\pageref{Bij}) of locally finite spaces.
Second, note that $e_\top \colon B_{i+1,j} \to B_{i,j}$ is induced in the same way from the face map $d_i \colon \mathcal{M}_{i+2+j} \to \mathcal{M}_{i+1+j}$, which is an inner face map, and is therefore finite since $\mathcal{M}$ is Möbius.
Similarly $d_0 \colon B_{i,j+1} \to B_{i,j}$ is obtained from $d_{i+1} \colon \mathcal{M}_{i+2+j} \to \mathcal{M}_{i+1+j}$ which is also an inner face map.
Finally the fibres of $e_\top^{(n)}$ are empty for $n$ sufficiently large because the fibres of $d_{i-n+1} \circ \dots \circ d_i \colon \mathcal{M}_{i+2+j} \to \mathcal{M}_{i+2-n+j}$ are empty for $n$ sufficiently large since $\mathcal{M}$ is Möbius. Similarly, the fibres of $d_0^{(n)}$ are empty for $n$ sufficiently large.
}
\end{proof}

\revised{
\subsection{The Möbius function of the incidence algebra of the decomposition space of finite posets}

We come back to the bicomodule configuration $\ds{B}$ of layered sets and posets given in Section~\ref{sec:layeredsetsposets}.

\begin{prop}[{\cite[Proposition~3.3]{Ca:mdrs}}]
    The bicomodule configuration $\ds{B}$ is Möbius.
\end{prop}

\begin{rmk}
This Möbius bicomodule configuration does not come from an adjunction of decomposition spaces.
\end{rmk}

We can now apply the generalised Rota formula of Theorem~\ref{rotaformulabicomod} to compute the Möbius function of the bialgebra of finite posets. The result is known \cite{ABS}, but its derivation from the generalised Rota formula is new and interesting.

\begin{thm}[{\cite[Theorem~3.4]{Ca:mdrs}}]\label{formula}
    The Möbius function of the incidence algebra of the decomposition space $\ds{C}$ of finite posets is 
    \[
    \mu(P) = 
    \begin{cases}
       (-1)^n &\text{ if $P \in \ds{C}_1$ is a discrete poset with $n$ elements,}\\
         0    &\text{ else.}
    \end{cases}
    \]
\end{thm}

This result can be extended to the incidence algebra of any directed restriction species, including rooted forests, and free operads~\cite{Ca:mdrs}.
 }


\newpage
\begin{thebibliography}{10}

\bibitem{ABS}
M.~Aguiar, N.~Bergeron, and F.~Sottile.
\newblock Combinatorial {H}opf algebras and generalized {D}ehn–{S}ommerville
  relations.
\newblock {\em Compositio Mathematica}, 142(1):1--30, 2006.

\bibitem{AguiarFerrer}
M.~Aguiar and W.~{Ferrer Santos}.
\newblock Galois connections for incidence {H}opf algebras of partially ordered
  sets.
\newblock {\em Advances in Mathematics}, 151(1):71--100, 2000.

\bibitem{Aigner}
M.~Aigner.
\newblock {\em Combinatorial theory}.
\newblock Classics in Mathematics. Springer-Verlag Berlin Heidelberg, 1997.

\bibitem{AF}
D.~Ayala and J.~Francis.
\newblock Fibrations of {$\infty$}-categories.
\newblock Preprint, \arxiv{1702.02681}.

\bibitem{BD}
J.~C. Baez and J.~Dolan.
\newblock From finite sets to {F}eynman diagrams.
\newblock In {\em Mathematics unlimited---2001 and beyond}, pages 29--50.
  Springer, Berlin, 2001.

\bibitem{BOORS}
J.~E. {Bergner}, A.~M. {Osorno}, V.~{Ozornova}, M.~{Rovelli}, and C.~I.
  {Scheimbauer}.
\newblock 2-{S}egal sets and the {W}aldhausen construction.
\newblock {\em Topology and its Applications}, 235:445--484, 2018.

\bibitem{Ca:mdrs}
L.~Carlier.
\newblock {M}öbius functions of directed restriction species and free operads,
  via the generalised {R}ota formula.
\newblock Preprint, \arxiv{1812.09915}.

\bibitem{CF}
P.~Cartier and D.~Foata.
\newblock {\em Problèmes combinatoires de commutation et réarrangements}.
\newblock Lecture Notes in Mathematics, No. 85. Springer-Verlag, Berlin-New
  York, 1969.

\bibitem{CLL}
M.~Content, F.~Lemay, and P.~Leroux.
\newblock Catégories de {M}öbius et fonctorialités : un cadre général pour
  l'inversion de {M}öbius.
\newblock {\em Journal of Combinatorial Theory, Series A}, 28(2):169--190,
  1980.

\bibitem{DK}
T.~Dyckerhoff and M.~Kapranov.
\newblock Higher {S}egal spaces {I}.
\newblock Preprint, \arxiv{1212.3563}. To appear in \emph{Springer Lecture Notes}.

\bibitem{GHK}
D.~Gepner, R.~Haugseng, and J.~Kock.
\newblock $\infty$-operads as analytic monads.
\newblock Preprint, \arxiv{1712.06469}.

\bibitem{GKT:HLA}
I.~Gálvez-Carrillo, J.~Kock, and A.~Tonks.
\newblock Homotopy linear algebra.
\newblock {\em Proceedings of the Royal Society of Edinburgh: Section A Mathematics}, 148(2):293--325, 2018.


\bibitem{GKT1}
I.~Gálvez-Carrillo, J.~Kock, and A.~Tonks.
\newblock Decomposition spaces, incidence algebras and {M}öbius inversion {I}: basic theory.
\newblock {\em Advances in Mathematics}, 331:952--1015, 2018.

\bibitem{GKT2}
I.~Gálvez-Carrillo, J.~Kock, and A.~Tonks.
\newblock Decomposition spaces, incidence algebras and {M}öbius inversion {II}: completeness, length filtration, and finiteness.
\newblock {\em Advances in Mathematics}, 333:1242--1292, 2018.

\bibitem{GKT:restrict}
I.~Gálvez-Carrillo, J.~Kock, and A.~Tonks.
\newblock Decomposition spaces and restriction species.
\newblock {\em International Mathematics Research Notices}, 2018.
\newblock
  \href{https://academic.oup.com/imrn/advance-article-abstract/doi/10.1093/imrn/rny089/5095270}{doi:10.1093/imrn/rny089}.

\bibitem{GKT:combinatorics}
I.~Gálvez-Carrillo, J.~Kock, and A.~Tonks.
\newblock Decomposition spaces in combinatorics.
\newblock Preprint, \arxiv{1612.09225}.

\bibitem{Illusie}
L.~Illusie.
\newblock {\em Complexe cotangent et déformations {II}}, volume 283 of {\em
  Lecture Notes in Mathematics, No. 85}.
\newblock Springer-Verlag, Berlin-New York, 1972.

\bibitem{Joyal02}
A.~Joyal.
\newblock Quasi-categories and {K}an complexes.
\newblock {\em Journal of Pure and Applied Algebra}, 175(1):207--222, 2002.
\newblock Special Volume celebrating the 70th birthday of Professor Max Kelly.

\bibitem{Joyal08}
A.~Joyal.
\newblock The theory of quasi-categories and its applications, 2008.
\newblock Available at
  \url{http://mat.uab.cat/~kock/crm/hocat/advanced-course/Quadern45-2.pdf}.

\bibitem{JoyalCL}
A.~Joyal.
\newblock Distributors and barrels, 2012.
\newblock Available at
  \href{https://ncatlab.org/joyalscatlab/published/Distributors+and+barrels}{Joyal's
  CatLab}, or
  \url{http://mat.uab.cat/~kock/crm/hocat/Joyal-Distributors-and-barrels.pdf}.

\bibitem{Leroux76}
P.~Leroux.
\newblock Les catégories de {M}öbius.
\newblock {\em Cahiers de topologie et géométrie différentielle},
  16:280--282, 1975.

\bibitem{Lurie}
J.~Lurie.
\newblock {\em {Higher Topos Theory}}.
\newblock Annals of Mathematics Studies. Princeton University Press, 2009.
\newblock Available at \url{http://www.math.harvard.edu/~lurie/papers/HTT.pdf}.

\bibitem{LurieHA}
J.~Lurie.
\newblock {Higher Algebra}, 2017.
\newblock Available at \url{http://www.math.harvard.edu/~lurie/papers/HA.pdf}.

\bibitem{Penney}
M.~D. {Penney}.
\newblock Simplicial spaces, lax algebras and the 2-{S}egal condition.
\newblock Preprint, \arxiv{1710.02742}.

\bibitem{Rota}
G.-C. Rota.
\newblock On the foundations of combinatorial theory {I}. {T}heory of {M}öbius
  functions.
\newblock {\em Zeitschrift für Wahrscheinlichkeitstheorie und Verwandte
  Gebiete}, 2:340--368, 1964.

\bibitem{Stanley1}
R.~P. Stanley.
\newblock {\em Enumerative combinatorics. {V}ol. {I}}.
\newblock The Wadsworth \& Brooks/Cole Mathematics Series. Wadsworth \&
  Brooks/Cole Advanced Books \& Software, Monterey, CA, 1986.

\bibitem{Walde}
T.~Walde.
\newblock Hall monoidal categories and categorical modules.
\newblock Preprint, \arxiv{1611.08241}.

\bibitem{Young}
M.~B. {Young}.
\newblock Relative $2$-{S}egal spaces.
\newblock {\em Algebraic and Geometric Topology}, 18:975--1039, 2018.

\end{thebibliography}

\address
\end{document}